\documentclass[12pt,reqno,a4paper]{article}

\usepackage{euscript}
\usepackage{amsmath,amssymb,amsfonts,amsthm}
\usepackage{mathtools}
\usepackage{dsfont}
\usepackage[colorlinks=true, linkcolor=blue]{hyperref}
\newtheoremstyle{special}%
{}%
{}%
{}%
{}%
{\scshape}%
{.}%
{.5em}%
{}

\newtheorem{theorem}{Theorem}
\newtheorem{proposition}[theorem]{Proposition}
\newtheorem{lemma}[theorem]{Lemma}
\newtheorem{cor}[theorem]{Corollary}
\newtheorem{definition}[theorem]{Definition}

\theoremstyle{special}
\newtheorem{remark}[theorem]{Remark}
\newtheorem{example}[theorem]{Example}

\renewcommand{\epsilon}{\varepsilon}
\renewcommand{\L}{\mathcal{L}}

\def\N{\mathbb{N}}

\def\R{\mathbb{R}}
\def\B{\mathcal{B}}

\def\W{\mathcal W}
\def\S{\mathcal{S}}

\DeclareMathOperator{\var}{Var}
\DeclareMathOperator{\esssup}{esssup}
\DeclareMathOperator{\essinf}{essinf}
\DeclareMathOperator{\spn}{span}

\title{Quenched limit theorems for expanding on average cocycles}
\date{\today}

\author{Davor Dragi\v cevi\'c\footnote{Department of Mathematics, University of Rijeka, Croatia. \texttt{E-mail:ddragicevic@math.uniri.hr}}, Julien Sedro\footnote{Universit\'e de Paris and Sorbonne Universit\'e, CNRS, Laboratoire de Probabilit\'es, Statistique et Mod\'elisation, F-75013 Paris, France. \texttt{Email:sedro@lpsm.paris}}}

\begin{document}
	\maketitle 

\begin{center}
Dedicated to the memory of Vlado Dragi\v cevi\' c (1945-2021)
\end{center}
	
	\begin{abstract}
		We prove quenched versions of a central limit theorem, a large deviations principle as well as a local central limit theorem for expanding on average cocycles. This is achieved by building an appropriate modification of the spectral method for nonautonomous dynamics developed
		by Dragi\v cevi\' c et al. (Comm Math Phys 360: 1121--1187, 2018), to deal with the case of random dynamics that exhibits nonuniform decay of correlations, which are ubiquitous in the context of the multiplicative ergodic theory. 
		Our results provide an affirmative answer to a question posed by Buzzi (Comm Math Phys 208: 25-54, 1999).
	\end{abstract}
	\section{Introduction}
	In the recent years, a great deal of effort has been devoted to the investigation of the statistical properties of \emph{random} dynamical systems. Indeed, those are key tools to model many natural phenomena, including the transport in complex environments such as in the ocean or the atmosphere~\cite{Arnold}: it is therefore crucial to understand their long term quantitative behavior. Among many remarkable contributions, we particularly emphasize those dealing with
	the decay of correlations~\cite{ABR, BB, BBR, BBMD, Buzzi, DL}, various (quenched or annealed) limit laws~\cite{ANV, ALS, DPZ, D, D1, D2, DH, H1, H3, HK,K2, NTV, OS, Su}, as well as recent results devoted to the linear response of random dynamical systems~\cite{BRS, DS, RS}. For similar results in the closely related context of sequential dynamical systems, we refer to~\cite{B1, B2, HNTV, KL, LS1, LS2, H2, NSV} and references therein.

	\subsection{Spectral method for limit theorems} 
	The Nagaev-Guivarc'h spectral method is a powerful approach to establish limit theorems for dynamical systems. Originally introduced in the context of Markov chains~\cite{N57, N61} and for the purpose of establishing 
	the central limit theorem in autonomous dynamics~\cite{GH88}, the method was broadly extended to cover many types of limit theorems: large deviations, Berry-Esseen estimates, local central limit theorems or almost-sure invariance principle, to name a few (we refer to the excellent survey \cite{G}, and references therein, for more details).
	\par\noindent In \cite{D} (see also~\cite{HK} for a similar approach), this method was further extended to the non-autonomous case of random compositions of expanding maps, allowing to obtain fiberwise (or \emph{quenched}) limit theorems, and relying on the careful study of the Lyapunov spectrum, and  the associated Oseledets spaces, of a so-called \emph{twisted transfer operator cocycle}. More precisely, given a measurable family $(T_\omega)_{\omega\in\Omega}$ of 
	(say) uniformly expanding maps of the unit interval $I$, parametrized by a probability space $(\Omega,\mathcal F,\mathds P)$ endowed with an invertible, measure-preserving and ergodic transformation  $\sigma:\Omega\circlearrowleft$, one forms the cocycle above $\sigma$:
	\begin{equation}\label{eq:cocycle}
		T_\omega^n:= T_{\sigma^{n-1}\omega}\circ\dots\circ T_\omega \quad \omega \in \Omega, \ n\in \N, 
	\end{equation}
	which is known to admit a family of equivariant measures $(\mu_\omega)_{\omega\in\Omega}$ (i.e. $\mu_\omega\circ T_{\omega}^{-1}=\mu_{\sigma\omega}$ for $\mathds P$-a.e $\omega\in\Omega$).
	Given a suitable observable $\psi:\Omega\times I\to\R$, one wants to study the asymptotic behavior of the random process $(\psi_{\sigma^n\omega}\circ T^n_\omega)_{n\ge 0}$ for fixed $\omega$ in a subset of full measure of $\Omega$, where $\psi_\omega=\psi(\omega, \cdot)$.  In this endeavor, one associates  to each transformation $T_\omega$, the  corresponding 
	\emph{transfer operator} $\L_\omega:=\L_{T_\omega}$.\footnote{i.e., for each $f\in L^1(I)$, $g\in L^\infty(I)$, $\int_I f\cdot g\circ T_\omega dm=\int_I \L_\omega f\cdot g dm$, where $m$ is the Lebesgue measure} In addition, one considers the twisted transfer operator $\L_\omega^\theta$, with twist parameter $\theta\in\mathbb C$, defined for $\phi\in L^1(I)$ by:
	\begin{equation}
		\L_\omega^\theta(\phi):=\L_\omega(e^{\theta\psi_\omega}\phi).
	\end{equation}
	\par\noindent The spectral method for random dynamics may then be crudely summarized in three key steps:
	\begin{enumerate}
		\item Relate the characteristic functions of the (random) Birkhoff sums $S_n\psi_\omega$, $n\in\mathbb N$ (see \eqref{bs}) to the $n$-th iterate of the twisted transfer cocycle.
		\item Establish the quasi-compactness (on a suitable Banach space) of the twisted transfer operator cocycle $(\L_\omega^\theta)_{\omega \in \Omega}$  for all parameters $\theta$ sufficiently close to $0$.
		\item Study the regularity w.r.t. the twist parameter $\theta$ of the top Lyapunov exponent $\Lambda(\theta)$ of the cocycle $(\L_\omega^\theta)_{\omega \in \Omega}$.
	\end{enumerate}
	Once established, these properties  allow to prove various (quenched) limit theorems, such as large deviations or (local) central limit theorem. 
	\par\noindent More recently, this spectral approach was used to establish quenched limit theorems for random hyperbolic dynamics \cite{D2} and random $U(1)$-extension of one-dimensional expanding maps \cite{CNW}.
	\subsection {Main contributions of the present paper} 
	Despite those developments, several natural and interesting types of random systems remained out of reach of the  method presented in~\cite{D}; most notably, systems exhibiting \emph{non-uniform in $\omega$} decay of correlations such as the  (general case of) random dynamical systems expanding on average  studied by Buzzi~\cite{Buzzi}.
	\par\noindent In the setting when all maps $T_\omega$ are chosen in a vicinity of a single (deterministic and transitive) map $T$ (which is precisely the setup of~\cite{CNW, D2}), the associated random system will exhibit uniform (in $\omega$) decay of correlations (with respect to suitable observables). However, if this is not the case, requiring the presence of a uniform decay of correlations becomes a restrictive condition (see~\cite[Section 2.2]{D1}).
	\\In sharp contrast to~\cite{D}, where contracting behavior in the family $(T_\omega)_{\omega\in\Omega}$ is only allowed on a set of zero measure (see \cite[Eq. (20)]{D}), the results in the present paper cover the general case of expanding on average random dynamics, that doesn't exhibit the  uniform in $\omega$ decay of correlations, providing a positive answer to a problem posed in~\cite[p.30]{Buzzi}. We emphasize that, compared with the precise class of expanding on average random dynamical systems introduced in~\cite{Buzzi}, we only need the additional mild assumption that the map $\omega \mapsto T_\omega$ has a countable range: this is to ensure that we can apply a suitable version of the multiplicative ergodic theorem~\cite{FLQ2}.\footnote{Allowing to choose among uncountably many maps would require to work with observables living in a separable space, e.g. fractional Sobolev spaces. In this aspect, we are limited by our choice to work with observables of bounded variation, which form a non-separable Banach space; in turns, this choice has other advantages.} 
	\\We note that there are previous works devoted to limit theorems for random dynamical system that allow contracting behavior on large measure sets \cite{ANV}, but only under the condition that the family $(T_\omega)_{\omega\in\Omega}$ only takes finitely many values (and assuming  that $(\Omega, \mathcal F, \mathbb P, \sigma)$ is a Bernoulli shift),  
	or don't require the presence of a uniform decay of correlations, such as the work of Kifer~\cite{K2} (partially inspired by the work of Cogburn~\cite{C}), as well as the first author and Hafouta~\cite{DH2}.
	Roughly speaking, the main idea in those papers is to pass to the associated induced system, where the inducing is done with respect to the region of $\Omega$ on which one has the uniform decay of correlations. Provided that we have a good control for the first entry time to this region, one can 
	deduce limit theorems for the original system from that of the induced one. The problematic aspect of this approach is the following: it is hard to ensure good control of the first entry time since a priori we cannot say anything about the region with respect to which we induce.
	Indeed, it remains an interesting open problem to build concrete examples to which the results of~\cite{K2, DH2} are applicable, that don't exhibit the uniform decay of correlations. 
	
	\par\noindent
	In the present paper we rather modify the approach introduced in~\cite{D}. The main technical novelty is the introduction of suitable ``adapted norms'' which depend on the random parameter $\omega$, and allow us to ``absorb the non-uniformity'' present in the constants appearing in the  decay of correlations estimate. The precise construction relies on the Oseledets splitting of the (unperturbed) transfer cocycle, and is reminiscent  of the classical construction due to Pesin for non-uniformly hyperbolic dynamics (see also \cite[Chapter 4]{Arnold} for random product of matrices). 
	\\Naturally, this will require some restrictions for the class of observables to which our limit theorems will apply. Indeed, we demand that the BV-norm of our observable $\psi_\omega$ is dominated by a tempered random variable $K$, that reflects the non-uniformity in the decay of correlations and is explicitly constructed from so-called ``Oseledets regularity functions", which are notoriously hard to compute: this can make our condition difficult to check in practice.
	\par\noindent Once the construction of adapted norms is carried out, the previously described steps follow by adapting classical perturbation techniques, based on the analytic version of the implicit function theorem, to our setting. In particular, this greatly simplify (when compared to~\cite{D}) the study of the regularity w.r.t. the twist parameter of the top Lyapunov exponent and the associated top Oseledets space generator. Apart from this, our proofs follow closely previous works, most notably \cite{D}.
	
	\subsection{Statement of main theorems} 
	We now turn to the statement of our main results. Given a probability space $(\Omega,\mathcal F,\mathds P)$, endowed with an invertible, measure-preserving ergodic transformation $\sigma:\Omega\circlearrowleft$, we let $(T_\omega)_{\omega\in\Omega}$ be a measurable family of maps, with the associated transfer operator cocycle $(\L_\omega)_{\omega\in\Omega}$ satisfying the requirements of Definition \ref{good}. The maps $T_\omega$ act on a set $X$ endowed with a probability measure $m$ and a notion of variation $\var$ as in Section \ref{sec:setup}. Then, there is (see Theorem \ref{7}) an equivariant family $(\mu_\omega)_{\omega \in \Omega}$ of absolutely continuous probability measures, $d\mu_\omega= v_\omega^0dm$ , where $v_\omega^0\in BV$.
	\par\noindent For an observable $\psi:\Omega\times X\to \R$, satisfying $\psi_\omega\in BV$, 
	\begin{equation}\label{ftr}  
	\quad  \esssup_{\omega\in\Omega}\|\psi_\omega\|_{BV}<+\infty 
	\end{equation} 
	and a centering condition \eqref{center}, we introduce the \emph{scaled observable} $\psi_K:\Omega\times X\to\R$, defined by
	\begin{equation}\label{scale}
		\psi_K(\omega,\cdot):=\frac{1}{K^2(\omega)}\psi(\omega,\cdot),
	\end{equation}	
	where the random variable $K$ is given by \eqref{LN}. We will be concerned with the asymptotic behavior of the Birkhoff sum associated to $\psi_K$, $S_n\psi_K(\omega,\cdot)$, defined by 
	\begin{equation}\label{scalebs}
		S_n\psi_K(\omega,x):=\sum_{i=0}^{n-1}\psi_K(\sigma^i\omega,T^i_\omega).
	\end{equation}
	Our first main result is a quenched Central Limit Theorem: 
	\begin{theorem}[Central Limit Theorem]\label{thm:CLT}
		Under the previous assumptions, and if the variance $\Sigma^2$, given by \eqref{var2}, satisfies $\Sigma^2>0$, then for any bounded continuous $\phi:\R\to\R$, for $\mathds P$-a.e $\omega\in\Omega$, one has
		\begin{equation}
			\lim_{n\to\infty}\int_X\phi\left(\frac{S_n\psi_K(\omega,x)}{\sqrt n}\right)d\mu_\omega(x)=\int\phi d \mathcal N(0,\Sigma^2),
		\end{equation}
		where $\mathcal N(0,\Sigma^2)$ is the centered normal law with variance $\Sigma^2$.
	\end{theorem}
	We then prove a Large Deviations Theorem:
	\begin{theorem}[Large Deviations Theorem]\label{thm:LDT}
		Under the previous assumptions, there exists $\epsilon_0>0$ and a non-negative, continuous, strictly convex function $c:(-\epsilon_0,\epsilon_0)\to\R$, vanishing only at $0$ and such that for any $0<\epsilon<\epsilon_0$, and $\mathds P$-a.e $\omega\in\Omega$, 
		\begin{equation}
			\lim_{n\to\infty}\frac{1}{n}\log\mu_\omega\left(S_n\psi_K(\omega,\cdot)>n\epsilon\right)=-c(\epsilon)
		\end{equation}  
	\end{theorem}
	Finally, we establish the quenched Local Central Limit Theorem, under an aperiodicity condition:
	\begin{theorem}[Local Central Limit Theorem]\label{thm:LCLT}
		Under the previous assumptions, and if the \emph{aperiodicity condition} \eqref{ap} is satisfied, then for $\mathds P$-a.e $\omega\in\Omega$, every bounded interval $J\subset\R$, one has:
		\begin{equation}
			\lim_{n\to\infty}\sup_{s\in\R}\left|\Sigma\sqrt n\mu_\omega(s+S_n\psi_K(\omega,\cdot)\in J)-\frac{|J|}{\sqrt 2\pi}e^{-\frac{s^2}{2n\Sigma^2}} \right|=0.
		\end{equation}
	\end{theorem}
	We emphasize that the above results generalize those in~\cite{D}, where $K$ is a constant random variable (see Remark~\ref{r0}) and consequently one can take $\psi_K=\psi$. Let us also emphasize that, as the example discussed in~\cite[Appendix]{DHS} illustrate, for observables satisfying only condition~\eqref{ftr} the asymptotic variance \eqref{variance} may fail to exist: this shows that our `scale restriction' \eqref{scale} is necessary, and that renormalized limit theorems such as Theorems~\ref{thm:CLT}, \ref{thm:LDT} and~\ref{thm:LCLT} are the best one can hope for in general. 
 	We also note that examples of observables for which our limit theorems hold may be constructed without the above scaling procedure (see Example~\ref{EX} for details).
	\medskip

	 Finally, we would like to end this introduction by noting that our methods can be used to obtain other type of limit theorems (e.g. Berry-Esseen estimates), but to keep the length of this paper within reasonable bounds, we refrain to give a complete list of those (for the specific instance of Berry-Esseen estimate in our setting, we may argue exactly as in \cite[Section 4.4]{DH}). In fact, among the `classical' limit theorems, only the almost sure invariance principle required a  non-trivial addition to our approach, which was carried out in~\cite{DHS}.
	\section{Preliminaries}
	
	\subsection{Multiplicative ergodic theorem}
	
	We begin by recalling the notion of a (linear) cocycle.
	
	\begin{definition}
		A tuple $\mathcal{R}=(\Omega, \mathcal{F}, \mathbb P, \sigma, \mathcal B, \L)$ is said to be a \emph{linear cocycle} or simply a \emph{cocycle} if the following conditions hold:
		\begin{itemize}
			\item $(\Omega,\mathcal F,\mathbb P)$ is a probability space and  $\sigma \colon \Omega \to \Omega$ is an invertible ergodic measure-preserving transformation;
			\item $(\mathcal B, \|\cdot \|)$ is a Banach space and
			$\mathcal L\colon \Omega\to L(\mathcal B)$ is a family of bounded linear operators.
		\end{itemize}
		We say that $\mathcal L$ is the \emph{generator} of $\mathcal R$.
	\end{definition}
	
	\noindent We will often identify a cocycle $\mathcal R=(\Omega, \mathcal{F}, \mathbb P, \sigma, \mathcal B, \L)$ with its generator $\mathcal L$.
	Moreover, we will write $\mathcal L_\omega$ instead of $\mathcal L(\omega)$.

	\begin{definition}
		Let $\mathcal{R}=(\Omega, \mathcal{F}, \mathbb P, \sigma, \mathcal B, \L)$  be a cocycle. We say that $\mathcal R$ is $\mathbb P$-continuous if the following holds:
		\begin{itemize}
			\item $\Omega$ is a Borel subset of a separable, complete metric space and $\sigma$ is a homeomorphism;
			\item $\mathcal L$  is continuous on each of countably many Borel sets whose union is $\Omega$.
		\end{itemize}
	\end{definition}

	Let $\mathcal{R}=(\Omega, \mathcal{F}, \mathbb P, \sigma, \mathcal B, \L)$ be a $\mathbb P$-continuous cocycle such that
	\begin{equation}\label{int}
		\int_\Omega \log^+ \|\L_\omega \|\, d\mathbb P(\omega) <+\infty.
	\end{equation}
	For $\omega \in \Omega$ and $n\in \N$, set
	\[
	\L_\omega^n=\L_{\sigma^{n-1} \omega} \circ \ldots \circ  \L_{\sigma \omega} \circ \L_\omega.
	\]
	It follows from Kingman's subadditive ergodic theorem that, for $\mathbb P$-a.e. $\omega \in \Omega$, the following limits exists:
	\[
	\Lambda(\mathcal{R}) = \lim_{n\to \infty} \frac1n\log \| \L_\omega^n \|
	\]
	and 
	\[
	\kappa(\mathcal{R}) =  \lim_{n\to \infty} \frac1n\log ic( \L_\omega^n),
	\]
	where \[\text{ic}(A):=\inf\Big\{r>0 : \ A(B_{\mathcal B})~\text{admits a finite covering by balls of radius }r \Big\},\] $B_{\mathcal B}$ is the unit ball of $\mathcal B$, and 
	\[-\infty \le \kappa(\mathcal R) \le \Lambda (\mathcal R)<+\infty \]

	\begin{definition}
		We say that  $\Lambda (\mathcal R)$ is the \emph{top Lyapunov exponent} of a cocycle $\mathcal R$, and that $\kappa(\mathcal R)$ is the  \emph{index of the compactness} of $\mathcal R$.
		
	\end{definition}
	
	\begin{definition}
		Let $\mathcal{R}=(\Omega, \mathcal{F}, \mathds P, \sigma, \mathcal B, \L)$ be a $\mathds P$-continuous cocycle such that \eqref{int} holds. We say that $\mathcal R$ is \emph{quasi-compact} if $\kappa (\mathcal R)<\Lambda (\mathcal R)$.
	\end{definition}
	The following result gives sufficient conditions under which a cocycle is quasi-compact.
	\begin{lemma}\label{QC}
		Let $\mathcal{R}=(\Omega, \mathcal{F}, \mathds P, \sigma, \mathcal B, \L)$ be a $\mathds P$-continuous cocycle such that \eqref{int} holds. Furthermore, let  $(\mathcal B', |\cdot|)$ be a Banach space such that $\mathcal B\subset \mathcal B'$ and that the inclusion
		$(\mathcal B, \|\cdot\|) \hookrightarrow (\mathcal B',|\cdot|)$ is compact. Finally, assume the following:
		\begin{itemize}
			\item $\L_\omega$ can be extended continuously to $(\mathcal B', |\cdot|)$ for $\mathbb P$-a.e. $\omega \in \Omega$;
			\item there exist $N\in \N$ and  measurable functions $\alpha_\omega, \beta_\omega, \gamma_\omega: \Omega \to \R$ such that for $f\in \mathcal B$ and $\mathbb P$-a.e. $\omega \in \Omega$,
			\begin{equation}\label{1}
				\| \L_\omega^N f\| \leq \alpha_\omega \|f\| + \beta_\omega |f|
			\end{equation}
			and
			\begin{equation}\label{2}
				\| \L_\omega \|  \leq \gamma_\omega;
			\end{equation}
			\item we have that 
			\begin{equation}\label{3}
				\int_\Omega \log \alpha_\omega \, d\mathbb P(\omega) < N\Lambda(\mathcal {R}) \text{ and } \int_\Omega \log \gamma_\omega \, d\mathbb P(\omega)<\infty.
			\end{equation}
		\end{itemize}
		Then, \[\kappa(\mathcal R) \le \frac 1 N\int_\Omega \log \alpha_\omega \, d\mathbb P(\omega). \]  In particular, 
		$\mathcal R$ is quasi-compact. 
	\end{lemma}
	
	\begin{proof}
		If $N=1$, the desired conclusion follows from~\cite[Lemma 2.1.]{D}. In the general case, we consider the cocycle $\mathcal R^N=(\Omega, \mathcal{F}, \mathbb P, \sigma^N, \mathcal B, \bar{\L})$ whose generator $\bar{\L}$ is given by 
		$\bar{\L}_\omega:=\L_\omega^N$, $\omega \in \Omega$. It is easy to show that \begin{equation}\label{KL} \kappa (\mathcal R^N)=N\kappa (\mathcal R) \quad  \text{and} \quad \Lambda (\mathcal R^N)=N\Lambda (\mathcal R).\end{equation} It follows from our assumptions and~\cite[Lemma 2.1.]{D} that \par\noindent $\kappa(\mathcal R^N)\le \int_\Omega \log \alpha_\omega \, d\mathbb P(\omega)$, 
		which together with~\eqref{KL} implies the desired conclusion.
	\end{proof}
	
	We are now in a position to recall the version of the multiplicative ergodic theorem (MET) established in~\cite{FLQ2}.
	\begin{theorem}\label{MET}
		Let $\mathcal R=(\Omega,\mathcal F,\mathbb
		P,\sigma,\mathcal B,\mathcal L)$ be a  quasi-compact $\mathbb P$-continuous cocycle. Then, there exist  $1\le l\le \infty$ and a sequence of exceptional
		Lyapunov exponents
		\[ \Lambda(\mathcal R)=\lambda_1>\lambda_2>\ldots>\lambda_l>\kappa(\mathcal R) \quad \text{(if $1\le l<\infty$)}\]
		or  \[ \Lambda(\mathcal R)=\lambda_1>\lambda_2>\ldots \quad \text{and} \quad \lim_{n\to\infty} \lambda_n=\kappa(\mathcal R) \quad \text{(if $l=\infty$).} \]
		Furthermore,  for $\mathbb P$-a.e. $\omega \in \Omega$ there exists a unique splitting (called the \textit{Oseledets splitting}) of $\mathcal B$ into closed subspaces
		\begin{equation}\label{eq:splitting}
			\mathcal B=V(\omega)\oplus\bigoplus_{j=1}^l Y_j(\omega),
		\end{equation}
		depending measurably on $\omega$ and  such that:
		\begin{enumerate} 
			\item  For each $1\leq j \leq l$, $Y_j(\omega)$ is finite-dimensional (i.e. $m_j:=\dim Y_j(\omega)<\infty$),  $Y_j$ is equivariant i.e. $\L_\omega Y_j(\omega)= Y_j(\sigma\omega)$ and for every $y\in Y_j(\omega)\setminus\{0\}$,
			\[\lim_{n\to\infty}\frac 1n\log\|\mathcal L_\omega^ny\|=\lambda_j.\]
			\item
			$V$ is weakly  equivariant i.e. $\L_\omega V(\omega)\subseteq V(\sigma\omega)$ and
			for every $v\in V(\omega)$, \[\lim_{n\to\infty}\frac 1n\log\|\mathcal
			L_\omega^nv\|\le \kappa(\mathcal R).\]
		\end{enumerate}
	\end{theorem}
	\subsection{Dual cocycles}
	In the following, we will need to establish `parallel' properties for a cocycle and its dual; hence, we recall in this section some basic facts related to this notion. 
	\par\noindent Given a cocycle $\mathcal R=(\Omega, \mathcal{F}, \mathbb P, \sigma, \mathcal B, \L)$, one defines the \emph{adjoint cocycle} $\mathcal R^*=(\Omega, \mathcal{F}, \mathbb P, \sigma^{-1}, \mathcal B, \L^*)$, where $(\L^*)_\omega=\L_{\sigma^{-1}\omega}^*$, $\omega \in \Omega$. We will write $\L_\omega^*$ for $(\L^*)_\omega$, which is thus the adjoint operator of $\L_{\sigma^{-1}\omega}$. We notice that $\mathcal R^*$ is $\mathds P$-continuous if and only if $\mathcal R$ is.
	\par\noindent Furthermore, it is easy to see that $\Lambda(\mathcal R^*)=\Lambda (\mathcal R)$ and that $\kappa(\mathcal R)=\kappa(\mathcal R^*)$. In particular, $\mathcal R$ is quasi-compact if and only if $\mathcal R^*$ is.
	\par\noindent One also has a natural relation between the Oseledets splitting and Lyapunov exponents of $\mathcal R$ and those of $\mathcal R^*$. Recall that, given a subspace $S\subset\B$, the \emph{annihilator} $S^\circ$ of $S$ is defined by $\{\ell\in\B^*,~\ell(f)=0~\forall~f\in S\}$; similarly, given $S^*\subset\B^*$ we set $(S^*)^\circ:=\{f\in\B,~\ell(f)=0~\forall\ell\in S^*\}.$
	\begin{theorem}[\cite{D}, Cor. 2.5 and Lemma 2.6]\label{thm:dualcocycle}
		Under the assumption of Theorem \ref{MET}, the adjoint cocycle admits a unique measurable equivariant splitting 
		\begin{equation}\label{split*}
			\B^*=V^*(\omega)\oplus \bigoplus_{i=1}^l Y_i^*(\omega),
		\end{equation}
		with the same exponents $(\lambda_i)_{i\in\{1,\dots,l\}}$ and multiplicities $(m_i)_{i\in\{1,\dots,l\}}$ as $\mathcal R$.
		Furthermore, denoting by $H(\omega):=V(\omega)\oplus \bigoplus_{i=2}^l Y_i(\omega)$ \\(resp. $H^*(\omega):=V^*(\omega)\oplus \bigoplus_{i=2}^l Y_i^*(\omega)$), one has
		\begin{equation}
			Y_1^*(\omega)=H(\omega)^\circ~\text{and}\quad Y_1(\omega)=(H^*(\omega))^\circ.
		\end{equation}
	\end{theorem}
	
	\subsection{Cocycles expanding on average}\label{sec:setup}
	We begin by recalling the setup from~\cite{Buzzi}. Let $(X, \mathcal G)$ be a measurable space endowed with a probability measure $m$ and a notion of a variation $\var \colon L^1(X, m) \to [0, \infty]$ which satisfies
	the following conditions:
	\begin{enumerate}
		\item[(V1)] $\var (th)=\lvert t\rvert \var (h)$;
		\item[(V2)] $\var (g+h)\le \var (g)+\var (h)$;
		\item[(V3)] $\lVert h\rVert_{L^\infty} \le C_{\var}(\lVert h\rVert_1+\var (h))$ for some constant $1\le C_{\var}<\infty$;
		\item[(V4)] for any $C>0$, the set  $\{h\colon X \to \mathbb R: \lVert h\rVert_1+\var (h) \le C\}$ is $L^1(m)$-compact;
		\item[(V5)] $\var(1_X) <\infty$, where $1_X$ denotes the function equal to $1$ on $X$;
		\item[(V6)] $\{h \colon X \to \mathbb R_+: \lVert h\rVert_1=1 \ \text{and} \ \var (h)<\infty\}$ is $L^1(m)$-dense in
		$\{h\colon X \to \mathbb R_+: \lVert h\rVert_1=1\}$.
		\item[(V7)] for any $f\in L^1(X, m)$ such that $\essinf f>0$, we have \[\var(1/f) \le \frac{\var (f)}{(\essinf f)^2}.\]
		\item[(V8)] $\var (fg)\le \lVert f\rVert_{L^\infty}\cdot \var(g)+\lVert g\rVert_{L^\infty}\cdot \var(f)$.
		\item[(V9)] for $M>0$, $f\colon X \to [-M, M]$ measurable and  every $C^1$ function $h\colon [-M, M] \to \mathbb C$, we have
		$\var (h\circ f)\le \lVert h'\rVert_{L^\infty} \cdot \var(f)$.
	\end{enumerate} We define
	\[
	BV=BV(X,m)=\{g\in L^1(X, m): \var (g)<\infty \}.
	\]
	Then, $BV$ is a Banach space with respect to the norm
	\[
	\lVert g\rVert_{BV} =\lVert g\rVert_1+ \var (g).
	\]
	\begin{remark}
		In the rest of the paper we will use (V1)-(V9) without explicitly referring to those properties. In particular, we will often use that 
		\[
		\|fg\|_{BV} \le C_{var} \|f\|_{BV} \cdot \|g\|_{BV} \quad \text{for $f, g\in BV$,}
		\]
		which follows readily from (V3) and (V8).
	\end{remark}
	
	\begin{example}\label{ex:var}
		It is well-known (see~\cite[Section 2.2.]{D}) that in the case when $X=[0, 1]$ and $m$ is the Lebesgue measure, $\var \colon L^1(X, m) \to [0, \infty]$ given by 
		\begin{equation}\label{def:1Dvar}
		\var (g)=\inf_{h=g (mod \ m)} \sup_{0=s_0<s_1<\ldots <s_n=1}\sum_{k=1}^n \lvert h(s_k)-h(s_{k-1})\rvert \quad g\in L^1(X, m),
		\end{equation}
		satisfies properties (V1)-(V9). 
		
		In  higher dimensions, when $X\subset \R^n$, $n>1$ and $m$ is again the Lebesgue measure, $\var \colon L^1(X, m) \to [0, \infty]$ defined by 
		\begin{equation}\label{def:HDvar}
		\var (f)=\sup_{0<\epsilon \le \epsilon_0}\frac{1}{\epsilon^\alpha}\int_{\R^d}osc (f,  B_\epsilon (x)))\, dm(x),
		\end{equation}
		where
		\[
		osc (f, B_\epsilon (x))=\esssup_{x_1, x_2 \in B_\epsilon (x)}\lvert f(x_1)-f(x_2)\rvert,
		\]
		also fulfills conditions (V1)-(V9). We refer to~\cite[Section 2.2.]{D} and~\cite{Saussol} for details.
	\end{example}
	Let $(\Omega, \mathcal{F}, \mathds P, \sigma)$ be a probability space and $\sigma \colon \Omega \to \Omega$  an invertible ergodic measure-preserving transformation. Let  $T_{\omega} \colon X \to X$, $\omega \in \Omega$ be a collection of non-singular  transformations (i.e.\ $m\circ T_\omega^{-1}\ll m$ for each $\omega$) acting   on $X$.  Each transformation $T_{\omega}$ induces the corresponding transfer operator $\mathcal L_{\omega}$ acting on $L^1(X, m)$ and  defined  by the following duality relation
	\[
	\int_X(\mathcal L_{\omega} \phi)\psi \, dm=\int_X\phi(\psi \circ T_{\omega})\, dm, \quad \phi \in L^1(X, m), \ \psi \in L^\infty(X, m).
	\]
	Thus, we obtain a cocycle of transfer operators  $(\Omega, \mathcal F, \mathbb P, \sigma, L^1(X, m), \mathcal L)$ that we denote by $\L=(\L_\omega)_{\omega \in \Omega}$.
	\begin{definition}\label{good}
		A cocycle $\L=(\L_\omega)_{\omega \in \Omega}$ of transfer operators is said to be \emph{good} if the following conditions hold:
		\begin{itemize}
			\item $\Omega$ is a Borel subset of a separable, complete metric space and $\sigma$ is a homeomorphism. Moreover, $\L$ is $\mathbb P$-continuous;
			\item there exists a random variable $C\colon \Omega \to (0, +\infty)$ such that $\log C\in L^1(\Omega, \mathbb P)$ and
			\begin{equation}\label{weakLY}
				\|\L_\omega h\|_{BV}\le C(\omega) \|h\|_{BV}, \quad \text{for $\mathbb P$-a.e. $\omega \in \Omega$ and $h\in BV$;}
			\end{equation}
			\item there exist $N\in \N$ and random variables $\alpha^N, K^N \colon \Omega \to (0, +\infty)$ such that 
			\begin{equation}\label{qc}
				\int_\Omega \log  \alpha^N \, d\mathbb P <0, \quad \log K^N \in L^1(\Omega, \mathbb P)
			\end{equation}
			and, for $\mathds P$-a.e. $\omega \in \Omega$ and $h\in BV$,
			\begin{equation}\label{strongLY}
				\var(\L_\omega^N h) \le \alpha^N (\omega) \var(h)+ K^N(\omega) \|h \|_1;
			\end{equation}
			\item for each $a>0$ and $\mathbb P$-a.e. $\omega \in \Omega$, there exist random numbers $n_c(\omega)<+\infty$ and $\alpha_0(\omega), \alpha_1(\omega), \ldots$ such that for every $h\in \mathcal C_a$, 
			\[
			\essinf_x (\L_\omega^n h)(x) \ge \alpha_n\|h \|_1 \quad \text{for $n\ge n_c$,}
			\]
			where 
			\[
			C_a:=\{ h\in L^\infty (X,m): \text{$h\ge 0$ and $\var(h) \le a\|h\|_1$} \}.
			\]
		\end{itemize}
	\end{definition}
	
	\begin{remark}\label{rem:good}
		Let us make a few comments on those assumptions:
		\begin{itemize}
			\item We note that Definition~\ref{good} almost coincides with~\cite[Definition 1.1]{Buzzi}, the only difference being that the first requirement in Definition~\ref{good} is absent in~\cite[Definition 1.1]{Buzzi}. We stress that the first requirement in Definition~\ref{good} ensures that we can apply Theorem~\ref{MET} to any good cocycle of transfer operators.
			\item We have that $\L$ is $\mathbb P$-continuous whenever the map $\omega \mapsto T_\omega$ has a countable range $\{T_1, T_2, \ldots \}$, and for each $j$, $\{\omega \in \Omega: T_\omega=T_j \} \in \mathcal F$ (see~\cite[Section 4.1]{FLQ2}).
			\item Observe that~\eqref{weakLY} implies that $\L_\omega$ is a bounded operator on $BV$ for $\mathbb P$-a.e. $\omega \in \Omega$. Thus, from now on we will view $\L$ as a cocycle acting on $BV$  (and not $L^1(X,m)$).
		\end{itemize}
	\end{remark}
	
	Let us now describe examples of systems having good cocycles of transfer operators:
	\begin{example}[Lasota-Yorke cocycles]\label{ex:good}
		Consider $X=[0,1]$, endowed with Lebesgue measure and the classical notion of variation $\var$ given by \eqref{def:1Dvar}. We say that $T:X\to X$ is a piecewise monotonic non-singular map (p.m.n.s map for short) if the following conditions hold:
		\begin{itemize}
			\item T is piecewise monotonic, i.e. there exists a subdivision $0=a_0<a_1<\dots<a_N=1$ such that for each $i\in\{0,\dots,N-1\}$, the restriction $T_i=T_{|(a_i,a_{i+1})}$ is monotonic (in particular it is a homeomorphism on its image).
			\item T is non-singular, i.e. there exists $|T'|:[0,1]\to\R_+$ such that for any measurable $E\subset (a_i,a_{i+1})$, $m(T(E))=\int_E|T'|dm$.
		\end{itemize}
		The intervals $(a_i,a_{i+1})_{i\in\{0,\dots,N-1\}}$ are called the intervals of $T$. We also set $N(T):=N$ and $\lambda(T):=\essinf_{[0,1]}|T'|$.
		\\We consider a family $(T_\omega)_{\omega\in\Omega}$ of random p.m.n.s as above, and such that $T:\Omega\times [0,1]\to[0,1],~(\omega,x)\mapsto T_\omega(x)$ is measurable.
		Denoting $N_\omega=N(T_\omega)$ and $\lambda_\omega=\lambda(T_\omega)$, we assume that
		\begin{itemize}
			\item The map $\omega\mapsto\left(\var\left(\frac{1}{|T'_\omega|}\right),N_\omega,\lambda_\omega,a_1,\dots,a_{N_\omega-1}\right)$ is measurable.
			\item We have the following expanding-on-average property:
			\begin{equation}\label{hyp:exponaverage}
				\int_\Omega\log\lambda_\omega~d\mathds P(\omega)>0.
			\end{equation}
			\item The map $\log^+\left(\frac{N_\omega}{\lambda_\omega}\right)$ is integrable. 
			\item The map $\log^+\left(\var\left(\frac{1}{|T'|}\right)\right)$ is integrable.
			\item $T_\omega$ is covering, i.e. for any interval $I\subset[0,1]$, there exists a random number $n_c(\omega)>0$ such that for any $n\ge n_c$, one has
			\begin{equation*}
				\essinf_{[0,1]} \L^{(n)}_\omega(\mathds 1_I)>0.
			\end{equation*}
		\end{itemize}
		We will call a cocycle satisfying the previous assumptions an \emph{expanding on average Lasota-Yorke cocycle}. For a countably-valued measurable family $(T_\omega)_{\omega\in\Omega}$ of expanding on average Lasota-Yorke cocycle, Remark \ref{rem:good} and \cite[Section 1.2]{Buzzi} tells us that the associated family of transfer operators is good in the sense of Definition \ref{good}.
	\end{example}
	
	\begin{example}[Smooth expanding on average cocycles] 
		We describe here an example, inspired by Kifer \cite{K1}, of cocycles to which the previous definition apply. Consider $X=\mathbb S^1$ be the unit circle, endowed with the Lebesgue measure $m$ and the notion of variation given by $\var(\phi):=\int_X |\phi'|~dm=\|\phi'\|_{L^1}$.
		We consider a family of maps $(T_\omega)_{\omega\in\Omega}$, such that $\Omega$ is a Borel subset of a separable, complete metric space, endowed with a homeomorphism $\sigma$, and $T:\Omega\times X\to X$ is measurable and $T_\omega:=T(\omega,\cdot)$ is $C^r$, $r\ge 2$. Using the same notations as the previous example (note that here, $N_\omega=1$ a.s), we make the following assumptions on this family:
		\begin{enumerate}
			\item The map $\omega\in\Omega\mapsto \left(\int_X\frac{|T_\omega''|}{(T_\omega')^2} dm,\lambda_\omega\right)$ is measurable, where $\lambda_\omega=\essinf_{[0, 1]}|T_\omega'|$.
			\item The expanding on average property \eqref{hyp:exponaverage} holds.
			\item The map $\log\left(\int_X\frac{|T_\omega''|}{(T_\omega')^2} dm\right)$ is $\mathds P$-integrable.
		\end{enumerate} 
		We call a family $(T_\omega)_{\omega\in\Omega}$ satisfying the previous assumption a \emph{smooth expanding on average cocycle}.
		\par\noindent Let us show that the transfer operator cocycle $(\L_\omega)_{\omega\in\Omega}$ associated to a countably valued, smooth expanding on average cocycle is a good family of transfer operators in the sense of Definition \ref{good}.
		We first observe that the $\mathds P$-continuity property of $\L$ follows, as in the previous example, from Remark \ref{rem:good}. Second, we show that for the norm\footnote{This is of course the classical $(1,1)$ Sobolev norm.} $\|.\|_{W^{1,1}}$ given by $\|\phi\|_{W^{1,1}}:=\|\phi\|_{L^1}+\|\phi'\|_{L^1}$, the Lasota-Yorke inequalities \eqref{weakLY} and \eqref{strongLY} hold, for $N=1$.
		We start by recalling deterministic Lasota-Yorke inequalities: for a $C^2$ map $T\colon  X\to X$, such that $D(T):=\sup\frac{T''}{(T')^2}<\infty$ (note that $D(T)$ is measurable in $T$), the well-known formula 
		\begin{equation}\label{eq:difftransferop}
			(\L_T\phi)'=\L_T\left(\frac{\phi'}{T'}\right)+\L_T\left(\frac{T''}{(T')^2}\phi\right)
		\end{equation} 
		easily implies that (see also \cite[Lemma 29]{GS})
		\begin{equation}\label{eq:detstrongLY}
			\|(L_T\phi)'\|_{L^1}\le \frac{1}{\lambda(T)}\|\phi'\|_{L^1}+D(T)\|\phi\|_{L^1},
		\end{equation}
		where $\lambda (T)=\essinf_{[0,1]}|T'|$. 
		In the spirit of Buzzi \cite[Lemma 1.3]{Buzzi}, one has:
		\begin{lemma}\label{lemma:detweakLY}
			For a $C^2$ map $T:X\circlearrowleft$,
			\begin{equation}
				\int_X|(\L_T\phi)'|~dm\le C_0(T)\|\phi\|_{W^{1,1}}.
			\end{equation}
			with \[C_0(T):= 2\max\left(\frac{1}{\lambda(T)}, \int_X\frac{|T''|}{(T')^2}~dm\right)\]
		\end{lemma}
		\begin{proof}
			Starting from \eqref{eq:difftransferop}, one gets
			\begin{align*}
				&\int_X|(\L_T\phi)'|~dm\\&\le\frac{1}{\lambda(T)}\int_X|\phi'|dm+\left(\int_X\frac{|T''|}{(T')^2}~dm\right)\max|\phi|\\
				&\le \left(\frac{1}{\lambda(T)}+\int_X\frac{|T''|}{(T')^2}~dm\right)\|\phi'\|_{L^1}+\int_X\frac{|T''|}{(T')^2}~dm\|\phi\|_{L^1},
			\end{align*}
			which immediately gives the result. 
		\end{proof}
		Now, \eqref{weakLY} follows easily from the previous lemma and the definition of a smooth expanding on average cocycle. To obtain \eqref{strongLY}, one needs a little more work and once again, we follow closely \cite[p.33]{Buzzi}. We denote $\lambda_*=\int_\Omega\log \lambda_\omega d\mathds P$ and $C_0(\omega)=C_0(T_\omega)$.
		By integrability of $\log C_0(\omega)$ and $\log\lambda_\omega$, we may choose $\epsilon_0>0$ sufficiently small so that, for any measurable $E\subset\Omega$ with $\mathds P(E)<\epsilon_0$, 
		\[\int_E \log C_0(\omega)d\mathds P\le \frac{\lambda_*}{4}\quad\text{and}\quad \int_E \log \lambda_\omega d\mathds P\le \frac{\lambda_*}{4}.\]
		We now let $\Delta$ be a constant so large that the set $\Omega_b:=\{\omega\in\Omega,~D(\omega)>\Delta\}$ has measure smaller than $\epsilon_0$, where $D(\omega)=D(T_\omega)$.  We then set:
		\[\alpha(\omega):=
		\left\{
		\begin{aligned}
			&1/\lambda_\omega~\text{if}~\omega\not\in\Omega_b\\
			&C_0(\omega)~\text{otherwise}.
		\end{aligned}
		\right. 
		\]
		and $K(\omega):=\max(C_0(\omega),\alpha(\omega),\Delta,6)$. One then has: 
		\begin{itemize}
			\item if $\omega\not\in\Omega_b$, then $\alpha(\omega)=\frac{1}{\lambda_\omega}$ and $K(\omega)\ge\Delta\ge D(\omega)$. Hence, by \eqref{eq:detstrongLY}, \eqref{strongLY} holds in this case.
			\item If $\omega\in\Omega_b$, then $\alpha(\omega)=C_0(\omega)$ and $K(\omega)\ge C_0(\omega)$, hence by Lemma \ref{lemma:detweakLY}, \eqref{strongLY} also holds in this case. 
		\end{itemize}
		Finally, we observe that by our assumptions, $\alpha$ is log-integrable (from which $K$ is also log-integrable), with
		\begin{align*}
			\int_\Omega\log\alpha(\omega)d\mathds P(\omega)&= -\int_\Omega\log\lambda_\omega d\mathds P(\omega)+\int_{\Omega_b}\left(\log\lambda_\omega+\log C_0(\omega)\right)d\mathds P(\omega)\\
			&\le -\lambda_*+\frac{\lambda_*}{4}+\frac{\lambda_*}{4}\\
			&<0.
		\end{align*}
		This completes the proof of the Lasota-Yorke inequalities for smooth expanding on average cocycles. As for the last assertion in Definition \ref{good},  we note that is shown in~\cite[Example 6]{DHS} that for each non-trivial interval $I\subset X$, for $\mathds P$-a.e $\omega\in\Omega$, there is a $n_c:=n_c(\omega,I)<\infty$ such that for all $n\ge n_c$,
\[
				T_\omega^n(I)=X.
			\]
Now it remains to argue as in \cite[Claim p.32]{Buzzi}, taking into account \cite[Remark 0.1, 2.]{Buzzi}.
	\end{example}

	\begin{example}[Multidimensional piecewise affine maps.]
		Our abstract setup also covers multidimensional examples. The one we describe now is due to Buzzi \cite[Appendix B]{Buzzi}.
		\par\noindent Recall that a polytope in $\R^d$ is defined as the intersection of half-spaces. If $X\subset\R^d$, let $P$ be a finite collection of pairwise disjoints, open polytopes $A$ of $\R^d$, such that $Y=\bigcup_{A\in P} A$ is dense in $X$. We now let $f:Y\to X$ be such that for any $A\in P$, $f:A\to f(A)\subset X$ is the restriction of an affine map $f_A$ of $\R^d$: we say that $(X,P,f)$ is a piecewise affine map. We will also assume that each $f_A$ is invertible.
		\par\noindent We define the expansion rate of $f$, \[\lambda(f):=\inf_{x\in Y}\inf_{\|v\|=1}\|Df(x)\cdot v\|.\]
		We also recall that, given a polytope $A\subset\R^d$, we can define the $\epsilon$-multiplicity of its boundary $\partial A$ at $x\in X$, $\mathrm{mult}(\partial A,\epsilon,x)$, as the number of hyperplanes in $\partial A$ having non-empty intersection with $B(x,\epsilon)$ the ball of radius $\epsilon$ centered at $x$. We then set
		\begin{align*}
		\mathrm{mult}(\partial P,\epsilon)&:=\sup_{x\in X}\sum_{\underset{A\in P}{x\in\bar A}}\mathrm{mult}(\partial A,\epsilon,x)\\
		\mathrm{mult}(\partial P)&:=\lim_{\epsilon\to 0}\mathrm{mult}(\partial A,\epsilon).
		\end{align*}
		Finally we notice that there are some $\epsilon>0$ for which $\mathrm{mult}(\partial P,\epsilon)=\mathrm{mult}(\partial P)$. We denote by $\epsilon(\partial P)$ the supremum of such $\epsilon$.
		\par\noindent Given a probability space $(\Omega,\mathcal F,\mathds P)$, endowed as usual with an invertible, measure-preserving and ergodic self map $\sigma$, we consider countably-valued, measurable families of polytopes $(A_\omega)_{\omega\in\Omega}$ of $X\subset\R^d$ and affine maps $(f_{A_\omega})_{\omega\in\Omega}$: this data defines a cocycle of random piecewise affine map $(X,P_\omega,f_\omega)$, for which we assume:
		\begin{enumerate}
			\item For any $n\in\mathbb N$, the map $\omega\mapsto(\lambda(f^n_\omega),|P^n_\omega|,\mathrm{mult}(\partial P^n_\omega),\epsilon(\partial P^n_\omega))$ is measurable.
			\item The map $\frac{|P|}{\lambda}$ is $\log^+$ $\mathds P$-integrable.
			\item The following expansion on average condition holds:
			\begin{equation*}
			\Lambda:=\lim_{n\to\infty}\lim_{K\to\infty}\int_\Omega\frac{1}{n}\log\min\left(\frac{\lambda^n_\omega}{\mathrm{mult}(P^n_\omega)},K\right)d\mathds P>0.
			\end{equation*}
			\item The following random covering condition holds: 
			\\For any ball $B\subset X$, $\mathds P$-a.e $\omega\in\Omega$ there is a $n_c:=n_c(\omega,B)$ such that $f_\omega^n(B)=X$ (modulo a null set for Lebesgue measure) for $n\ge n_c$,
		\end{enumerate}
		Under the previous assumptions, and for the notion of variation on $X$ given by \eqref{def:HDvar}, it is established in \cite[Prop B.1]{Buzzi} that a random piecewise affine map has a good  random transfer operator, in the sense of \cite[Def.1.1]{Buzzi}. Together with the assumption that this map is countably valued, this shows that the associated transfer operator cocycle is good in the sense of Definition \ref{good}.
	\end{example}

	\begin{remark}[Multidimensional piecewise expanding maps, beyond the affine case]
		The recent paper \cite{BGT}, which studies statistical properties of so-called `random Saussol maps', a class of piecewise expanding on average multidimensional systems, opens up the possibility  to apply our approach to this much larger family of multidimensional examples, by establishing \eqref{weakLY} and \eqref{strongLY} for the transfer operator cocycle acting on spaces on bounded oscillation. Despite this progress, the question of mixing remains open in this setting, so that we cannot directly apply our approach to this type of system.
	\end{remark}
	For any cocycle having a good family of transfer operators in the sense of Definition \ref{good} (in particular, for the three previous examples), the following result holds:
	\begin{theorem}[\cite{Buzzi}, Main Theorem]\label{7}
		Let $\L=(\L_\omega)_{\omega \in \Omega}$ be a good cocycle of transfer operators. Then, we have the following:
		\begin{enumerate}
			\item there exists an essentially unique measurable family $(v_\omega^0)_{\omega \in \Omega}\subset BV$ such that $v_\omega \ge 0$, $\int_X v_\omega^0 \, dm=1$ and
			\[
			\mathcal L_\omega v_\omega^0=v_{\sigma \omega}^0, \quad \text{for $\mathbb P$-a.e. $\omega \in \Omega$.}
			\]
			\item There exists a random variable $Z\colon \Omega \to (0, +\infty)$ and $\rho \in (0,1)$ such that for $\mathds P$-a.e $\omega\in\Omega$, any $\phi\in BV$ and any $n\ge 0$,
			\begin{equation}\label{decayofcorrelations}
				\left\|\L_\omega^n\phi-\left(\int_X\phi~dm\right)v^0_{\sigma^n\omega}\right\|_{\infty}\le Z(\omega)\rho^n\|\phi\|_{L^1}.
			\end{equation}
			\item Furthermore, 
			\begin{equation}\label{decayofcorrelations1}
				\bigg{\lvert} \int_X \mathcal L_\omega^n (\varphi v_\omega^0)\psi \, dm \bigg{\rvert} \le Z(\omega)\rho^n \|\psi \|_\infty \cdot \|\varphi \|_{BV},
			\end{equation}
			for $\mathbb P$-a.e. $\omega \in \Omega$, $\psi$ bounded and $\varphi \in BV$ such that $\int_X \varphi\, d\mu_\omega=0$, where $d\mu_\omega=v_\omega^0\, dm$.
		\end{enumerate}
	\end{theorem}
	\begin{remark}
		The above stated Theorem~\ref{7}, particularly its second assertion,  is slightly different from the one stated in \cite{Buzzi}: let us explain how to get from the former to the latter:
		\begin{itemize}
			\item \cite[Main Theorem, 2.]{Buzzi} is stated at $\sigma^{-n}\omega$: to get it at $\omega$ instead, we simply remark that we may take $p=0$ in \cite[Prop. 4.1]{Buzzi} (instead of $p=-n$), which is possible since the proposition is valid for any $|p|\le n$.
			\item \cite[Main Theorem, 2.]{Buzzi} is stated for non-negative $\phi$ such that $\|\phi\|_1=1$. Hence, to get our version, starting from $\phi\ge 0$, simply replace $\phi$ by $\frac{\phi}{\|\phi\|_1}$ and multiply both sides by $\|\phi\|_1$. 
			For arbitrary real-valued\footnote{We note that, by decomposing it as the sum of its real and imaginary part, this also holds for general complex-valued $\phi \in BV$.} $\phi\in BV$, decompose it as the sum of its positive and negative parts: $\phi=\phi^+-\phi^-$, and using the inequality $\|u-v\|_\infty\le\|u+v\|_\infty$ for non-negative, bounded functions, we get:
			\begin{align*}
				\left\|\L_\omega^n\phi-\left(\int_X\phi~dm\right)v^0_\omega\right\|_{\infty}&\le 	\left\|\L_\omega^n|\phi|-\left(\int_X|\phi|~dm\right)v^0_\omega\right\|_{\infty}
			\end{align*}
			which immediately gives the announced result.  
		\end{itemize} 
	\end{remark}
	From Theorem \ref{7}, one easily derives the following:
	\begin{lemma}\label{4}
		Let $\L=(\L_\omega)_{\omega \in \Omega}$ be a good cocycle of transfer operators. Then, $\Lambda(\L) \ge 0$ and $\kappa(\L)<0$. In particular, $\L$ is quasi-compact. 
	\end{lemma}
	
	\begin{proof}
		Using the same notation as in the statement of Theorem~\ref{7}, we have that 
		\[
		\begin{split}
			\limsup_{n\to \infty} \frac 1 n \log \|\L_\omega^n \|_{BV}& \ge \limsup_{n\to \infty} \frac 1 n \log \|\L_\omega^n v_\omega^0\|_{BV} \\
			&=\limsup_{n\to \infty} \frac 1 n \log \|v_{\sigma^n \omega}^0  \|_{BV} \\
			&\ge \limsup_{n\to \infty} \frac 1 n \log \|v_{\sigma^n \omega}^0  \|_{1} \\
			&=0,
		\end{split}
		\]
		for $\mathbb P$-a.e. $\omega \in \Omega$. This implies that $\Lambda (\L)\ge 0$. 
		
		We now apply Lemma~\ref{QC} with $\| \cdot \|=\|\cdot \|_{BV}$ and $|\cdot |=\| \cdot \|_1$. Indeed, observe that~\eqref{weakLY} and~\eqref{strongLY}  imply that~\eqref{1} and~\eqref{2} hold with
		\[
		\alpha_\omega=\alpha^N(\omega), \quad \beta_\omega=K^N(\omega)+1 \quad \text{and} \quad \gamma_\omega=C(\omega). 
		\]
		Since $\log C\in L^1(\Omega, \mathbb P)$, $\int_\Omega \log  \alpha^N \, d\mathbb P <0$ and $\Lambda (\L)\ge 0$, we have that~\eqref{3} holds.  By Lemma~\ref{QC}, we obtain the desired conclusion.
		
	\end{proof}
	
	\begin{proposition}\label{zero}
		Let $\L=(\L_\omega)_{\omega \in \Omega}$ be a good cocycle of transfer operators. Then, $\Lambda (\L)=0$.
	\end{proposition}
	
	\begin{proof}
		By Lemma~\ref{4} we have that $\Lambda (\L) \ge 0$. 
		By using that $\|\L_\omega \|_1 \le 1$ and applying~\cite[Lemma C.5]{GTQ} (which we can due to Lemma~\ref{4}) for the cocycle $(\L_\omega^N)_{\omega \in \Omega}$ over $\sigma^N$, we conclude  that $\Lambda (\mathcal L)\le 0$.
		
	\end{proof}
	
	\begin{proposition}\label{90}
		Let $\L=(\L_\omega)_{\omega \in \Omega}$ be a good cocycle of transfer operators. Furthermore, set
		\[
		BV^0=\bigg{\{} h\in BV: \int_X h\, dm=0 \bigg{\}}.
		\]
		Then, using the same notation as in the statement of Theorem~\ref{MET},  we have the following:
		\begin{itemize}
			\item For $\mathbb P$-a.e. $\omega \in \Omega$,
			\begin{equation}\label{split}
				BV=Y_1(\omega) \oplus BV^0, 
			\end{equation}
			where $Y_1(\omega)$ is a one-dimensional subspace of $BV$ spanned by $v_\omega^0$.
			\item For $\mathbb P$-a.e. $\omega \in \Omega$, 
			\begin{equation}
				BV^0=V(\omega)\oplus\bigoplus_{j=2}^l Y_j(\omega).
			\end{equation}
		\end{itemize}
	\end{proposition}
	
	\begin{proof}
		We claim that 
		\begin{equation}\label{neg}
			\lim_{n\to \infty} \frac 1 n \log \|\L_\omega^n h\|_{BV} <0, \quad \text{for $\mathbb P$-a.e. $\omega \in \Omega$ and $h\in BV^0$.}
		\end{equation}
		Once we establish~\eqref{neg}, both assertions of the proposition will follow from  the uniqueness of the Oseledets splitting \eqref{eq:splitting} and the simple observation that for any $h\in BV$ and $\mathbb P$-a.e. $\omega \in \Omega$,
		\[
		h= \bigg (\int_X h\, dm \bigg )v_\omega^0+\bigg (h-\bigg (\int_X h\, dm \bigg )v_\omega^0 \bigg )\in \spn (v_\omega^0)  +BV^0.
		\]
		\par\noindent We first observe that the existence of the limit in~\eqref{neg} follows from Theorem~\ref{MET}. 
		\par\noindent For $h\in BV^0$, one has from~\eqref{decayofcorrelations} that
		\begin{equation}\label{162}
			\limsup_{n\to\infty}\frac{1}{nN}\log\|\L_\omega^{nN} h\|_{\infty}\le \log \rho <0.
		\end{equation}
		Choose $a>0$ sufficiently small such that $a+N\log \rho <0$ and 
		\[
		a+\int_\Omega \log\alpha^N \, d\mathbb P<0.
		\]
		By applying~\cite[Lemma C.5]{GTQ} for the cocycle $(\bar{\L}_\omega)_{\omega \in \Omega}$ over $\sigma^N$ where $\bar{\L}_\omega=e^a \L_\omega^N$, we conclude that from~\eqref{162} that
		\[
		\limsup_{n\to \infty} \frac{1}{nN} \log \|\L_\omega^{nN} h\|_{BV} <0,
		\]
		which implies~\eqref{neg}.
	\end{proof}
	
	Notice that Theorem \ref{thm:dualcocycle}, Lemma \ref{4}, Propositions \ref{zero} and \ref{90} imply the following result for the dual cocycle $\L^*$:
	\begin{cor}
		If $\L$ is a cocycle of good transfer operators, then $\L^*$ satisfies:
		\begin{enumerate}
			\item $0=\Lambda(\L^*)>\kappa(\L^*)$. In particular, $\L^*$ is quasi-compact.
			\item One has the following splitting for $BV^*$:
			\[BV^*=\spn (m)\oplus Y_1(\omega)^\circ,\]
			where $\spn (m)$ is the one-dimensional subspace of $BV^*$ spanned by Lebesgue measure\footnote{we identify the Lebesgue measure $m$ with the functional $\phi \to \int_X\phi \,  dm$ on $BV$} and $Y_1(\omega)^\circ$ is the annihilator of $v_\omega^0$.
		\end{enumerate}
	\end{cor}
	
	We now recall the notion of a tempered random variable. 
	\begin{definition}
		Let $K \colon \Omega \to (0, +\infty)$ be a measurable map. We say that $K$ is \emph{tempered} if
		\[
		\lim_{n\to \pm \infty} \frac 1 n\log K(\sigma^n \omega)=0, \quad \text{for $\mathbb P$-a.e. $\omega \in \Omega$.}
		\]
	\end{definition}
	
We will need the  following classical result (see~\cite[Proposition 4.3.3.]{Arnold}).
	\begin{proposition}\label{PA}
		Let $K \colon \Omega \to (0, +\infty)$ be a tempered random variable. For each $\epsilon >0$, there exists a tempered random variable $K_\epsilon \colon \Omega \to (1, +\infty)$ such that
		\[
		\frac{1}{K_\epsilon (\omega)} \le K(\omega) \le K_\epsilon (\omega) \quad \text{and} \quad K_\epsilon(\omega)e^{-\epsilon |n|} \le K_\epsilon (\sigma^n \omega) \le K_\epsilon (\omega) e^{\epsilon |n|},
		\]
		for $\mathbb P$-a.e. $\omega \in \Omega$ and $n \in \mathbb Z$.
	\end{proposition}
	
	\begin{proposition}\label{91}
		Let $\L=(\L_\omega)_{\omega \in \Omega}$ be a good cocycle of transfer operators. Furthermore, let $\Pi(\omega) \colon BV \to BV^0$, $\omega \in \Omega$ be a family of projections corresponding to the splitting~\eqref{split}, and let $\lambda_2<0$ be the second Lyapunov exponent of $\L$. For any $\epsilon\in(0,-\lambda_2)$, there exist  tempered random variables $D_i \colon \Omega \to [1, +\infty)$, $\in \{1, 2\}$ such that:
		\begin{itemize}
			\item for $\mathbb P$-a.e. $\omega \in \Omega$, $n\ge 0$ and $\phi \in BV$,
			\begin{equation}\label{436}
				\|\mathcal L_\omega^n \Pi(\omega) \phi \|_{BV } \le D_1(\omega)e^{-\lambda n} \|\phi\|_{BV},
			\end{equation}
			where we set $\lambda=-\lambda_2-\epsilon>0$;
			\item for $\mathbb P$-a.e. $\omega \in \Omega$, $n\ge 0$ and $\phi \in BV$,
			\begin{equation}\label{437}
				\|\mathcal L_\omega^n (I-\Pi(\omega)) \phi \|_{BV } \le D_2(\omega)e^{\epsilon n} \|\phi\|_{BV};
			\end{equation}
		\end{itemize}
	\end{proposition}
	
	\begin{proof}
		This is a direct consequence of Propositions~\ref{zero}, \ref{90} and~\cite[Proposition 3.2.]{BD}. Indeed, the proof of~\cite[Proposition 3.2]{BD} implies that we can take 
\begin{equation}\label{D111}
D_1(\omega)= (1+\|v_\omega^0\|_{BV} )\sup_{n\ge 0} (\|\L_\omega^n \rvert_{BV^0} \|_{BV} e^{\lambda n}),
\end{equation}
and
\begin{equation}\label{D22}
D_2(\omega)=
\sup_{n\ge 0} (\|v_{\sigma^n \omega}^0\|_{BV} e^{-\epsilon n}).
\end{equation}
	\end{proof}

\begin{remark}
	Observe that:
	\begin{itemize}
		\item One can replace $D_1$ in~\eqref{436} and $D_2$ in~\eqref{437} by $D=\max \{D_1, D_2\}$. However, in order to construct adapted norms, it is convenient to work with $D_1$ and $D_2$ separately.
		\item For $\mathbb P$-a.e. $\omega \in \Omega$, 
		\begin{equation}\label{D2}
		D_2(\omega) e^\epsilon \le D_2(\sigma \omega).
		\end{equation}
	\end{itemize}
\end{remark}
	
	\section{Twisted transfer operator cocycles}
	
	\subsection{Adapted norms}
	
	Throughout this section, we take a good cocycle of transfer operators $\L=(\L_\omega)_{\omega \in \Omega}$. Let $\epsilon, \lambda>0$, $D_i\colon \Omega \to [1, +\infty)$, $i=1,2$  and $\Pi(\omega)$, $\omega \in \Omega$
	be as in statement of Proposition~\ref{91}. 
	For $\mathbb P$-a.e. $\omega \in \Omega$ and $\phi \in BV$, let
	\[
	\begin{split}
		\|\phi \|_\omega &=\sup_{n\ge 0}(e^{\lambda n}\|\mathcal L_\omega^n \Pi(\omega) \phi \|_{BV })+\frac{1}{D_2(\omega) }\sup_{n\ge 0} (e^{-\epsilon n}\|\mathcal L_\omega^n (\phi- \Pi(\omega) \phi) \|_{BV }) \\
		&\phantom{=}+\bigg |\int_X \phi \, dm \bigg |.
	\end{split}
	\]
Observe that $\omega \mapsto \| \phi \|_\omega$ is measurable for each $\phi \in BV$. 
	The main properties of adapted norms are collected in the following result.
	\begin{proposition}\label{AN}
		The following holds:
		\begin{enumerate}
			\item There exists a tempered random variable $K \colon \Omega \to [1, +\infty)$ such that, for $\mathbb P$-a.e. $\omega \in \Omega$ and $\phi \in BV$, 
			\begin{equation}\label{LN}
				\frac{1}{K(\omega)} \|\phi \|_{BV} \le \| \phi \|_\omega \le K(\omega) \| \phi\|_{BV}.
			\end{equation}
			\item For $\mathbb P$-a.e. $\omega \in \Omega$ and $\phi \in BV$,
			\begin{equation}\label{an1}
				\|\phi \|_1 \le \|\phi\|_\omega.
			\end{equation}
			\item For $\mathbb P$-a.e. $\omega \in \Omega$ and $\phi \in BV$,
			\begin{equation}\label{an2}
				\|\mathcal L_\omega^n \Pi(\omega) \phi \|_{\sigma^n \omega } \le e^{-\lambda n} \|\phi\|_{\omega}.
			\end{equation}
			\item There exists a constant $C'>0$ such that for $\mathbb P$-a.e. $\omega \in \Omega$ and $\phi \in BV$:
			\begin{equation}\label{an3}
				\|\mathcal L_\omega \phi \|_{\sigma \omega} \le C'\|\phi \|_\omega.
			\end{equation}
			\item We have that 
			\begin{equation}\label{an4}
				\esssup \|v_\omega^0 \|_{\omega}<+\infty.
			\end{equation}
		\end{enumerate}
	\end{proposition}
	
	\begin{proof}
		It follows from~\eqref{436}, \eqref{437} and $\| \cdot \|_1 \le \| \cdot \|_{BV}$ that
		\[
		\|\phi \|_\omega \le (D_1(\omega)+2) \|\phi \|_{BV},
		\]
		for $\mathbb P$-a.e. $\omega \in \Omega$ and $\phi \in BV$.
		Moreover, 
		\[
		\|\phi \|_{\omega} \ge \|\Pi(\omega)\phi \|_{BV}+\frac{1}{D_2(\omega)}\|\phi-\Pi(\omega)\phi\|_{BV} \ge \frac{1}{D_2(\omega)}\|\phi \|_{BV},
		\]
		for $\mathbb P$-a.e. $\omega \in \Omega$ and $\phi \in BV$. From these two estimates we conclude that~\eqref{LN} holds with 
		\begin{equation}\label{K11}
		K(\omega)=\max \{D_1(\omega)+2, D_2(\omega) \},
		\end{equation}
		which is clearly tempered (since $D_1$ and $D_2$ are tempered).
		
	Furthermore, noting that $\| \phi- \Pi(\omega) \phi \|_1=\left|\int_X \phi \, dm\right|$,  for $\mathbb P$-a.e. $\omega \in \Omega$ and $\phi \in BV$ we have that 
		\[
		\begin{split}
			\|\phi \|_\omega &\ge \sup_{n\ge 0}(\|\mathcal L_\omega^n \Pi(\omega) \phi \|_{BV } \cdot e^{\lambda n})+\| \phi- \Pi(\omega) \phi \|_1 \\
			&\ge \|\Pi(\omega)\phi \|_{BV}+\| \phi- \Pi(\omega) \phi \|_1 \\
			&\ge \|\Pi(\omega)\phi \|_{1}+\| \phi- \Pi(\omega) \phi \|_1  \\
			&\ge \|\phi\|_1, 
		\end{split}
		\]
		which yields~\eqref{an1}.
		
		Next, observe that 
		\[
		\begin{split}
			\|\mathcal L_\omega^n \Pi(\omega) \phi \|_{\sigma^n \omega } &=\sup_{m\ge 0}(\|\mathcal L_{\sigma^n \omega}^m \mathcal L_\omega^n \Pi(\omega) \phi \|_{BV}e^{\lambda m}) \\
			&=\sup_{m\ge 0}(\| \mathcal L_\omega^{n+m}\Pi(\omega) \phi \|_{BV}e^{\lambda m}) \\
			&=e^{-\lambda n} \sup_{m\ge 0}(\| \mathcal L_\omega^{n+m} \Pi(\omega) \phi \|_{BV}e^{\lambda (m+n)}) \\
			&\le e^{-\lambda n} \|\phi\|_{\omega},
		\end{split}
		\]
		for $\mathbb P$-a.e. $\omega \in \Omega$ and $\phi \in BV$. Consequently, \eqref{an2} holds. 
		
		In addition, for $\mathbb P$-a.e. $\omega \in \Omega$ and $\phi \in BV$ one has (using~\eqref{D2}) that 
		\[
		\begin{split}
			\|\mathcal L_\omega \phi \|_{\sigma \omega} &=\sup_{n\ge 0}(\|\mathcal L_\omega^{n+1} \Pi(\omega) \phi \|_{BV } \cdot e^{\lambda n})\\
			&\phantom{=}+\frac{1}{D_2(\sigma \omega)}\cdot\sup_{n\ge 0} (\|\mathcal L_\omega^{n+1} (\phi- \Pi(\omega) \phi) \|_{BV } \cdot e^{-\epsilon n})\\
			&\phantom{=}+\bigg | \int_X \L_\omega \phi  \, dm \bigg | \\
			&\le e^{-\lambda}\sup_{n\ge 0}(\|\mathcal L_\omega^{n+1} \Pi(\omega) \phi \|_{BV } \cdot e^{\lambda (n+1)})\\
			&\phantom{\le}+e^{-\epsilon}\frac{1}{D_2(\omega)}\cdot\sup_{n\ge 0} (\|\mathcal L_\omega^{n+1} (\phi- \Pi(\omega) \phi) \|_{BV } \cdot e^{-\epsilon (n+1)})\\
			&\phantom{=}+\bigg| \int_X \phi \, dm\bigg | \\
			&\le \|\phi \|_\omega,
		\end{split}
		\]
		and therefore~\eqref{an3} holds. 
		
		Finally, \eqref{D22} implies that 
		\[
		\begin{split}
			\|v_\omega^0\|_\omega &=\frac{1}{D_2(\omega) } \cdot\sup_{n\ge 0} (\|\mathcal L_\omega^n v_\omega^0 \|_{BV } \cdot e^{-\epsilon n})+\bigg | \int_X v_\omega^0 \, dm  \bigg | \\
&=\frac{1}{D_2(\omega) } \cdot\sup_{n\ge 0} (\| v_{\sigma^n \omega}^0 \|_{BV } \cdot e^{-\epsilon n})+1 \\
&=2,
		\end{split}
		\]
	for $\mathbb P$-a.e. $\omega \in \Omega$,	 which yields~\eqref{an4}. The proof of the proposition is completed. 
	\end{proof}
	
	\begin{remark}
		Our construction of adapted norms is somewhat similar to the classical construction due to Pesin (see for example~\cite[Section 3]{BD}). However, our construction is different since besides transforming~\eqref{436} and~\eqref{437} into uniform type conditions (see~\eqref{an2} and~\eqref{an3}), an important feature of our adapted norms is given by properties~\eqref{an1} and~\eqref{an4}. As we will see later on, these properties will be  of central importance for our arguments. 
	\end{remark}
	
	\begin{remark}\label{formula}
By putting~\eqref{D111} and~\eqref{D22} in~\eqref{K11}, we obtain an explicit expression for the random variable $K$. 
	\end{remark}
	Let us now introduce  ``dual adapted norms'': we set, for $\ell\in BV^*$ and $\mathbb P$-a.e. $\omega \in \Omega$,
	\begin{equation}\label{def:dualadaptednorm}
		\|\ell\|^*_\omega:=\inf\{C>0: |\ell(\phi)|\le C\|\phi\|_\omega \ \text{for $\phi \in BV$}\}.
	\end{equation}
Observe that $\omega \mapsto \|\ell \|_\omega^*$ is measurable for each $\ell \in BV^*$. 
	The dual adapted norms $\|\cdot \|_\omega^*$ satisfy properties similar to the adapted norms $\|\cdot \|_\omega$. We summarize them in the following proposition.
	\begin{proposition}
		The following holds:
		\begin{enumerate}
			\item For $\mathbb P$-a.e $\omega\in\Omega$ and $\ell\in BV^*$,
			\begin{equation}\label{DLN}
				\frac{1}{K(\omega)}\|\ell\|_{BV^*}\le\|\ell\|_\omega^*\le K(\omega)\|\ell\|_{BV^*},
			\end{equation}
			with $K(\omega)$ is as in \eqref{LN}.
			\item For  $\mathbb P$-a.e $\omega\in\Omega$ and  $\ell\in L^\infty(X)$,\footnote{we identify $\ell$ with the functional $\phi \to \int_X\phi\ell \, dm$ on $BV$}
			\begin{equation}\label{dan1}
				\|\ell\|_\omega^*\le \|\ell\|_\infty.
			\end{equation}
			\item  For $\mathbb P$-a.e $\omega\in\Omega$ and  $\ell\in BV^*$,
			\begin{equation}\label{dan3}
				\|\L_\omega^*\ell\|^*_\omega \le C'\|\ell\|_{\sigma\omega}^*,
			\end{equation}
			where $C'$ as in~\eqref{an3}.
			\item For $\mathbb P$-a.e $\omega\in\Omega$,  $\ell\in BV^*$ and  $n\ge 0$,
			\begin{equation}\label{dan2}
				\|(\L_\omega^n)^*\Pi^*(\sigma^n\omega)\ell\|^*_{\omega}\le e^{-\lambda n}\|\ell\|_{\sigma^n\omega}^*,
			\end{equation}
			where $\lambda$ is given by Proposition \ref{91} and $\Pi^*(\omega)\ell:=\ell-\ell(v^0_\omega)m$.
			\item The Lebesgue measure $m$ satisfies, for $\mathbb P$-a.e $\omega\in\Omega$
			\begin{equation}\label{eq:goodLebesgue}
				\|m\|^*_\omega\le 1.
			\end{equation}
			In particular, $\esssup_{\omega\in\Omega}\|m\|_\omega^*<+\infty.$
		\end{enumerate}
	\end{proposition}	
	
	\begin{proof}
		Using~\eqref{LN} we have that for $\mathbb P$-a.e $\omega\in\Omega$, $\ell\in BV^*$ and $\phi \in BV$,
		\[
		|\ell (\phi) | \le \| \ell \|_{BV^*} \cdot \| \phi \|_{BV}\le K(\omega) \| \ell \|_{BV^*} \cdot \| \phi \|_{\omega},
		\]
		which implies the second inequality in~\eqref{DLN}. The first inequality can be  established similarly.
		
		Moreover, using~\eqref{an1} we have that 
		\[
		|\ell (\phi) | \le \| \ell \|_\infty \cdot \| \phi \|_1 \le \| \ell \|_\infty \cdot \| \phi \|_\omega, 
		\]
		for $\mathbb P$-a.e $\omega\in\Omega$, $\ell\in L^\infty(X)$ and $\phi \in BV$. Hence, \eqref{dan1} holds.
		
		Furthermore, it follows from~\eqref{an3} that 
		\[
		|\L_\omega^*\ell(\phi) |=|\ell (\L_\omega \phi )| \le \| \ell \|_{\sigma \omega}^* \cdot \| \L_\omega \phi \|_{\sigma \omega} \le C'\| \ell \|_{\sigma \omega}^* \cdot  \| \phi \|_\omega,
		\]
		for $\mathbb P$-a.e $\omega\in\Omega$, $\ell\in BV^*$ and $\phi \in BV$. We conclude that~\eqref{dan3} holds. 
		
		On the other hand, by~\eqref{an2} we have that 
		\begin{align*}
			| (\L_\omega^n)^*\Pi^*(\sigma^n\omega)\ell (\phi) | &=|\ell(\L_\omega^n \Pi(\omega) \phi)|\\  
			&\le \| \ell \|_{\sigma^n \omega}^* \cdot \| \L_\omega^n \Pi(\omega) \phi \|_{\sigma^n \omega}\\ 
			&\le e^{-\lambda n}\| \ell \|_{\sigma^n \omega}^* \cdot \| \phi \|_\omega, 
		\end{align*}
		for a.e $\omega\in\Omega$, $\ell\in BV^*$ and $\phi \in BV$. Thus, we proved~\eqref{dan2}.
		
		Finally, \eqref{an1} implies that 
		\[
		|m(\phi)| =\bigg | \int_X \phi \, dm \bigg | \le \| \phi \|_1 \le \| \phi \|_\omega, 
		\]
		for $\mathbb P$-a.e. $\omega \in \Omega$ and $\phi \in BV$ which implies~\eqref{eq:goodLebesgue}.
	\end{proof}
	
	\subsection{Twisted transfer operators}
	We start by fixing an observable. Let $\psi \colon \Omega \times X \to \R$ be a measurable map such that the following conditions hold:
	\begin{itemize}
		\item $\psi_\omega:=\psi(\omega, \cdot ) \in BV$ for $\omega \in \Omega$;
		\item we have that \begin{equation}\label{eq:condobservable}
			\esssup_{\omega \in \Omega} \bigg (K( \omega)^2  \|\psi_\omega\|_{BV} \bigg )<+\infty, 
		\end{equation}
		where $K$ is given by Proposition~\ref{AN};
		\item for $\mathbb P$-a.e. $\omega \in \Omega$,
		\begin{equation}\label{center}
			\int_X \psi_\omega v_\omega^0 \, dm=0.
		\end{equation}
	\end{itemize}
	
	\begin{remark}\label{r0}
		Let us make a few comments on those assumptions:
		\begin{itemize}
			\item Since $K(\omega) \ge 1$, it follows  from~\eqref{eq:condobservable} that
			\begin{equation}\label{b}
				\esssup_{\omega \in \Omega} \|\psi_\omega \|_{BV} <+\infty.
			\end{equation}
			\item We observe that~\eqref{eq:condobservable} is closed under centering. More precisely, if $\psi$ satisfies~\eqref{eq:condobservable} than so does $\psi'$ given by
			\[
			\psi_\omega'=\psi_\omega - \int_X \psi \, d\mu_\omega \quad \omega \in \Omega.
			\]
			\item In the setting of~\cite{D}, \eqref{436} and~\eqref{437} hold with $D$ a constant (see~\cite[Definition 2.8]{D} and~\cite[Lemma 2.11.]{D}). Moreover, $\esssup_{\omega \in \Omega} \|v_\omega^0\|_{BV}<+\infty$. Consequently, it follows from the proof of Proposition~\ref{AN} that~\eqref{LN} holds with $K$ being a constant. Therefore, in this setting~\eqref{eq:condobservable} and~\eqref{b} are equivalent. 
		\end{itemize}
	\end{remark}
	Next, we describe examples of observables satisfying \eqref{eq:condobservable}.
	\begin{example} \label{EX}
		\begin{itemize}
			\item Let $\psi:\Omega\times X\to \R$ be  such that \eqref{center} and~\eqref{b} hold. Then the `rescaled' observable $\psi_K:=K^{-2}\psi$ satisfies \eqref{eq:condobservable} and~\eqref{center}.
			\item  Due to the somewhat complicated expression for $K$ (see Remark~\ref{formula}), it is of interest to describe alternative ways of constructing observables satisfying~\eqref{eq:condobservable}. In order to do so, 
take an arbitrary $\delta>0$. By Proposition~\ref{PA}, there exists a random variable $K_\delta \colon \Omega \to [1, +\infty)$ such that 
			\[
			K(\omega)\le K_\delta (\omega) \quad \text{and} \quad K_\delta (\sigma^{-n} \omega)\le e^{\delta n} K_\delta (\omega),
			\]
			for $\mathbb P$-a.e. $\omega \in \Omega$ and $n\in \N$.  Take an arbitrary $C>0$. Furthermore, choose $T>0$ such that $\mathbb P(\Omega_T')>0$, where
			\[
			\Omega_T':= \{ \omega \in \Omega : K_\delta (\omega) \le T \}. 
			\]
We note that this holds for all $T$ sufficiently large.
			Set $\Omega_T^0=\Omega_T'$ and 
			for $n\in \N$, let $\Omega_T^n=\sigma^{-n}(\Omega_T')\setminus \cup_{k=0}^{n-1} \sigma^{-k}(\Omega_T')$. Clearly, $\Omega_T^n \cap \Omega_T^m=\emptyset$ for $n\neq m$. Moreover, the ergodicity of $\sigma$ implies that 
			$\mathbb P(\cup_{n=0}^\infty \Omega_T^n)=1$. If we choose $\psi$ so that 
			\[
			\|\psi_\omega \|_{BV} \le \frac{C}{T^2} e^{-2\delta n} \quad \text{for $\omega \in \Omega_T^n$ and  $n\ge 0$,}
			\]
			we have that 
			\[
			\esssup_{\omega \in \Omega} \bigg (K( \omega)^2  \|\psi_\omega\|_{BV} \bigg ) \le C.
			\]
			In particular, \eqref{eq:condobservable} holds. By centering our observable $\psi$, we can ensure that~\eqref{center} also holds (see Remark~\ref{r0}). 
		\end{itemize}
	\end{example}

	For $\theta \in \mathbb C$ and $\omega \in \Omega$, set
	\[
	\L_\omega^\theta \phi=\L_\omega (e^{\theta \psi_\omega} \phi), \quad \phi \in BV.
	\]
	
	\begin{lemma}\label{24}
		There exists a constant $C''>0$ such that for $|\theta |\le 1$, $\mathbb P$-a.e. $\omega \in \Omega$ and $\phi \in BV$,
		\[
		\|\L_\omega^\theta \phi \|_{\sigma \omega} \le C'' \|\phi \|_\omega. 
		\]
	\end{lemma}
	
	\begin{proof}
		By~\eqref{LN} and~\eqref{an3}, we have that 
		\[
		\begin{split}
			\| \mathcal L_\omega^\theta \phi -\mathcal L_\omega \phi \|_{\sigma \omega} &=\|\mathcal L_\omega ( (e^{\theta \psi_ \omega}-1)\phi) \|_{\sigma \omega} \\
			&\le C'\|(e^{\theta \psi_\omega}-1)\phi\|_\omega \\
			&\le C'K(\omega) \|(e^{\theta \psi_\omega}-1)\phi\|_{BV} \\
			&\le C'C_{var} K(\omega) \|e^{\theta \psi_\omega}-1\|_{BV} \cdot \|\phi \|_{BV}.
		\end{split}
		\]
		On the other hand, the arguments in the proof of~\cite[Lemma 3.13]{D} give  that 
		\[
		\begin{split} 
			\|e^{\theta \psi_\omega}-1\|_{BV} &\le  e^{\|\psi_\omega\|_\infty} \bigg (\|\psi_\omega \|_\infty+\var(\psi_\omega) \bigg ) \\
			&\le (C_{var}+1)e^{\|\psi_\omega \|_\infty} \| \psi_\omega \|_{BV}, 
		\end{split}
		\]
		and therefore
		\[
		\| \mathcal L_\omega^\theta \phi -\mathcal L_\omega \phi \|_{\sigma \omega} \le C'C_{var}(C_{var}+1)K(\omega)^2e^{\|\psi_\omega \|_\infty} \| \psi_\omega \|_{BV} \cdot \|\phi\|_\omega. 
		\]
		From~\eqref{eq:condobservable} (and~\eqref{b}), we conclude that there exists a constant $N>0$ such that 
		\[
		\| \mathcal L_\omega^\theta \phi -\mathcal L_\omega \phi \|_{\sigma \omega} \le N\|\phi\|_\omega,
		\]
		for $|\theta|\le 1$, $\mathbb P$-a.e. $\omega \in \Omega$ and $\phi \in BV$. The conclusion of the lemma now follows readily from~\eqref{an3} and the triangle inequality.
	\end{proof}
	\begin{remark}
		From Lemma~\ref{24}, we conclude immediately that the dual twisted transfer operator acts boundedly, i.e. that
		\begin{equation*}
			\|(\L_{\omega}^\theta)^*\ell\|_{\omega}\le C''\|\ell\|_{\sigma\omega},
		\end{equation*}
		for $\mathbb P$-a.e $\omega\in\Omega$ and every $|\theta|\le1$, $\ell\in BV^*$, and with $C''$ as in Lemma \ref{24}.
	\end{remark}
	\subsection{An auxiliary regularity result}
	
	Let $\mathcal S$ consists of all measurable maps $\mathcal V\colon \Omega \times X\to \mathbb C$ such that $\mathcal V(\omega, \cdot) \in BV$ for $\omega \in \Omega$ and 
	$\|\mathcal V\|_{\mathcal S}:=\esssup_{\omega \in \Omega} \|\mathcal V(\omega, \cdot) \|_\omega <+\infty$. For $\mathcal V \in \mathcal S$, we will often write $\mathcal V_\omega$ instead of $\mathcal V(\omega, \cdot)$.  
	\par\noindent The proof of the following result is a straightforward consequence of~\eqref{LN} (and completeness of $BV$).
	\begin{lemma}
		We have that $(\mathcal S, \| \cdot \|_{\mathcal S})$ is a Banach space.
	\end{lemma}
	
	\begin{remark}
		It follows from~\eqref{an4} that the map $(\omega, x)\mapsto v_\omega^0(x)$ belongs to $\mathcal S$. From now on, we will denote this map by $v^0$.
	\end{remark}
	Let $\mathcal S_0$ denote the closed subspace of $\mathcal S$ which consists of all $\mathcal V\in \mathcal S$ such that
	\[
	\int_X \mathcal V(\omega, \cdot) \, dm=0, \quad \text{for $\mathbb P$-a.e. $\omega \in \Omega$.}
	\]
	Let $B_{\mathbb C} (0,1)$ denote the unit ball in $\mathbb C$. We define $G\colon B_{\mathbb C}(0,1) \times \mathcal S_0 \to \mathcal S$ by
	\[
	G(\theta, \mathcal W) (\omega, \cdot)=\mathcal L_{\sigma^{-1} \omega}^\theta (\mathcal W(\sigma^{-1} \omega, \cdot)+v_{\sigma^{-1} \omega}^0(\cdot)).
	\]
	We claim that $G$ is well-defined.  Indeed, Lemma~\ref{24} implies that 
	\[
	\|G(\theta, \mathcal W) (\omega, \cdot) \|_\omega \le C''(\|\mathcal W(\sigma^{-1} \omega, \cdot)\|_{\sigma^{-1} \omega}+\| v_{\sigma^{-1} \omega}^0 \|_{\sigma^{-1} \omega}),
	\]
	for $\mathbb P$-a.e. $\omega \in \Omega$. Therefore,  using~\eqref{an4} we conclude that 
	\[
	\esssup_{\omega \in \Omega}\|G(\theta, \mathcal W) (\omega, \cdot) \|_\omega <+\infty.
	\]
	Hence, $G$ is well-defined. We also define $H\colon B_{\mathbb C}(0,1) \times \mathcal S_0 \to L^\infty (\Omega)$ by 
	\[
	H(\theta, \mathcal W)(\omega)=\int_X \mathcal L_{\sigma^{-1} \omega}^\theta (\mathcal W(\sigma^{-1} \omega, \cdot)+v_{\sigma^{-1} \omega}^0(\cdot))\, dm.
	\]
	We note that $H$ is well-defined since $G$ is well-defined and due to~\eqref{an1}. Some finer properties of $G$ and $H$ are given in the following result.

	\begin{lemma}\label{U}
		There exists a neighborhood  $\mathcal U\subset  B_{\mathbb C} (0,1) \times \mathcal S_0$ of $(0,0)$ such that $G$ and $H$ are analytic on $\mathcal U$ and that $\essinf_{\mathcal U} \|H\|_{L^\infty}\ge \frac{1}{2}$.
	\end{lemma}
	\begin{proof}
		For any $\theta\in\mathbb C$, $|\lambda |<1$, $\mathbb P$-a.e. $\omega\in\Omega$ and $\mathcal W\in \mathcal S$, we have that 
		\begin{align*}
			G(\theta,\W)(\omega)&=\L_{\sigma^{-1}\omega}^\theta(\W(\sigma^{-1}\omega)+v^0_{\sigma^{-1}\omega})\\
			&=\L_{\sigma^{-1}\omega}(e^{\theta\psi_{\sigma^{-1}\omega}}\left(\W(\sigma^{-1}\omega)+v^0_{\sigma^{-1}\omega}\right))\\
			&=\sum_{k=0}^\infty\frac{\theta^k}{k!}\L_{\sigma^{-1}\omega}\left(\psi_{\sigma^{-1}\omega}^k(\W(\sigma^{-1}\omega)+v^0_{\sigma^{-1}\omega})\right),
		\end{align*}
where $\psi_{\sigma^{-1}\omega}^k$ denotes the $k$-th power of $\psi_{\sigma^{-1}\omega}$. 
		We now remark that the maps  $\phi_k \colon B_{\mathbb C}(0,1)\times\S_0 \to \S$ given by 
		\[\phi_k(\theta,\W):=\dfrac{\theta^k}{k!}\L_{\sigma^{-1}\omega}\left(\psi_{\sigma^{-1}\omega}^k(\W(\sigma^{-1}\omega)+v^0_{\sigma^{-1}\omega})\right) \] are analytic.
		Moreover, \eqref{LN} and~\eqref{an3} imply that 
		\begin{align*}
			&\left\|\L_{\sigma^{-1}\omega}\left(\psi_{\sigma^{-1}\omega}^k(\W(\sigma^{-1}\omega)+v^0_{\sigma^{-1}\omega})\right)\right\|_{\omega}\\
			&\le C'\|\psi_{\sigma^{-1}\omega}^k(\W(\sigma^{-1}\omega)+v^0_{\sigma^{-1}\omega})\|_{\sigma^{-1}\omega}\\
			&\le C' K(\sigma^{-1}\omega)C_{var} \|\psi_{\sigma^{-1}\omega}^k\|_{BV} \cdot \|\W(\sigma^{-1}\omega)+v^0_{\sigma^{-1}\omega}\|_{BV}\\
			&\le C' K(\sigma^{-1}\omega)^2 C_{var}^k \|\psi_{\sigma^{-1}\omega}\|_{BV}^k \cdot \left(\|\W\|_{\S}+\|v^0\|_{\S}\right)\\
			&\le C' [\esssup_{\omega\in\Omega}(C_{var} K(\sigma^{-1}\omega)^2\|\psi_{\sigma^{-1}\omega}\|_{BV})]^k \left(\|\W\|_{\S}+\|v^0\|_{\S}\right),
		\end{align*}
		where we used that  $K(\omega)^{2/k}\le K(\omega)^2$ (which holds since $K(\omega) \ge 1$). We also note that~\eqref{eq:condobservable} implies that 
		$\esssup_{\omega\in\Omega}(C_{var} K(\sigma^{-1}\omega)^2\|\psi_{\sigma^{-1}\omega}\|_{BV})<+\infty$. 
		In particular, we obtain that, for  $\rho \in (0,1)$ sufficiently small,  $k\mapsto \|\rho^k\phi_k(\theta,\W)\|_{\S}$ is summable, and thus  $G$ is a uniform limit of analytic maps on some neighborhood  $\mathcal U\subset  B_{\mathbb C} (0,1) \times \mathcal S_0$ of $(0,0)$.
		
		A similar argument can be applied to $H$. In particular, $H$ is continuous and thus by shrinking  $\mathcal U$ if necessary, 
		we can assume that for $(\theta,\W)\in U$ we have that 
		\[\|H(\theta,\W)-H(0,0)\|_{L^\infty}\le\frac{1}{4}. \]
		Since $H(0,0)(\omega)=\int_X\L_{\sigma^{-1}\omega}v^0_{\sigma^{-1}\omega}dm=1$ for $\mathbb P$-a.e. $\omega \in \Omega$, we easily obtain that \[ \essinf_\mathcal U \|H\|_{L^\infty}\ge \frac{1}{2}, \]
		which completes the proof of the lemma. 
	\end{proof}
	
	\begin{lemma}
		The map $F:\mathcal U \subset B_{\mathbb C}(0,1)\times\S_0\to\S_0$ given by
		\[F(\theta,\W):=\frac{G(\theta,\W)}{H(\theta,\W)}-\W-v^0,\]
		is well-defined, analytic and satisfies $F(0,0)=0$.
		Furthermore, its differential w.r.t $\W$ at $(0,0)$, $D_2F(0,0)\colon \S_0\to\S_0$ is invertible. 
	\end{lemma}
	\begin{proof}
		The first part of the lemma follows readily from Lemma~\ref{U}.
		Furthermore, note that 
		\begin{equation*}
			(D_2F(0,0)\W)_\omega=\L_{\sigma^{-1} \omega} \W_{\sigma^{-1} \omega}-\W_\omega,
		\end{equation*}
		for $\omega \in \Omega$ and $\W \in \S_0$.
		
		We now prove that $D_2F(0,0)$ is invertible. Assume that $\W \in \S_0$ satisfies that $D_2 F(0,0) \W=0$.  Hence, $\W_{\omega}=\mathcal L_{\sigma^{-1} \omega} \W_{\sigma^{-1} \omega}$ for $\mathbb P$-a.e. $\omega \in \Omega$. Consequently, it follows from~\eqref{an2} that 
		for each $n\in \mathbb N$ and $\mathbb P$-a.e. $\omega \in \Omega$, 
		\[
		\begin{split}
			\| \W_{\omega} \|_\omega &= \| \mathcal L_{\sigma^{-n} \omega}^n \W_{\sigma^{-n} \omega} \|_\omega  \\
			&\le e^{-\lambda n}\| \W_{\sigma^{-n} \omega} \|_{\sigma^{-n} \omega} \\
			&\le e^{-\lambda n} \| \W \|_{\S}.
		\end{split}
		\]
		Letting $n\to \infty$, we conclude that $\W=0$. Thus, $D_2 F(0,0)$ is injective. 
		
		In order to establish the surjectivity of $D_2F(0,0)$, let us take an arbitrary $\W \in \S_0$. We define $\mathcal V$ by 
		\[
		\mathcal V(\omega, \cdot)=-\sum_{n=0}^\infty \mathcal L_{\sigma^{-n} \omega}^n \W(\sigma^{-n} \omega, \cdot), \quad \omega \in \Omega.
		\]
		Using~\eqref{an2},  it is easy to show that $\mathcal V \in \S_0$. Moreover, a simple computation yields that $D_2F(0,0)\mathcal V=\W$. Together with the open mapping theorem, this completes the proof of the lemma. 
	\end{proof}

	The following is the main result of this subsection. 
	\begin{theorem}\label{thm:geneigenvec}
		There exists a neighborhood $U\subset B_{\mathbb C}(0,1)$ of $0\in\mathbb C$, such that for any $\theta\in U$, there exist $v^\theta\in\S$, $\lambda^\theta\in L^\infty(\Omega)$, satisfying:
		\begin{enumerate}
			\item The maps $U\ni \theta \mapsto v^\theta\in\S$ and $U\ni \theta \mapsto \lambda^\theta\in L^\infty(\Omega)$ are analytic.
			\item For any $\theta\in U$ and $\mathbb P$-a.e. $\omega\in\Omega$, $v_\omega^\theta$, $\lambda^\theta_\omega$ satisfy
			\begin{align}\label{eq:geneigenvec}
				\L_\omega^\theta v_\omega^\theta&=\lambda_\omega^\theta v_{\sigma\omega}^\theta\\
				\lambda_\omega^\theta&=\int_X\L_\omega^\theta v_\omega^\theta dm.
			\end{align}
		\end{enumerate}
	\end{theorem} 
	\begin{proof}
		By the previous lemma, we may apply the analytic implicit function theorem between Banach spaces to $F$, to construct a neighborhood $U\subset B_{\mathbb C}(0,1)$ of $0$, a $w^\theta\in\S_0$, analytic in $\theta$, and such that for any $\theta\in U$,
		\[0=F(\theta,w^\theta)=\frac{G(\theta,w^\theta)}{H(\theta,w^\theta)}-w^\theta-v^0.\]
		Setting
		\begin{equation}
			v^\theta:= v^0+w^\theta,
		\end{equation}
		one obtains the desired  result.
	\end{proof}	            
	
	For future use we also note the following:
	\begin{lemma}\label{lemma:derivatives}
		The first and second derivatives at $0$ of the analytic map $\theta\ni U'\mapsto \lambda^\theta\in L^\infty(\Omega)$ are given by:
		\begin{align*}
			\frac{d\lambda_\omega^\theta}{d\theta}(0)&=\int_X \psi_\omega v^0_\omega dm=0\\
			\frac{d^2\lambda_\omega^\theta}{d\theta^2}(0)&=\int_X \psi_\omega^2 v^0_\omega dm+ 2\sum_{j=1}^\infty\int_X \psi_\omega \cdot \left[\L^{j}_{\sigma^{-j}\omega}(\psi_{\sigma^{-j}\omega}v^0_{\sigma^{-j}\omega})\right]dm.
		\end{align*}
	\end{lemma}
	\begin{proof}
		Classical computation, see e.g. \cite[Lemma 4.5]{D}.
	\end{proof}
	Let us now explain how to get the dual of the previous construction: we define
	\begin{equation}
		\mathcal N:=\{\Phi:\Omega\to BV^*~\text{measurable},~\esssup_{\omega\in\Omega}\|\Phi_\omega\|_\omega^*<\infty\}.
	\end{equation}
	Similarly to $\mathcal S$, we may set, for $\Phi\in\mathcal N$, $\|\Phi\|_\mathcal N:=\esssup_{\omega\in\Omega}\|\Phi_\omega\|_\omega^*$, which defines a norm that turns $\mathcal N$ into a Banach space. We notice that by \eqref{eq:goodLebesgue}, $m\in\mathcal N$. Let us set
	\begin{equation}
		\mathcal N_0:=\{\Phi\in\mathcal N,~\Phi_\omega(v^0_\omega)=0 \quad \text{for $\mathbb P$-a.e. $\omega \in \Omega$} \}.	
	\end{equation}
	Then, $\mathcal N_0$ is a closed subspace of $\mathcal N$.
	
	We may now introduce $G^*:B_\mathbb C(0,1)\times\mathcal N_0\to \mathcal N$ and $H^*:B_\mathbb C(0,1)\times\mathcal N_0\to L^\infty(\Omega)$ by \[G^*(\theta,\Phi)(\omega):=(\L^\theta_{\omega})^* \left(\Phi_{\sigma\omega}+m\right),\]
	and 
	\[H^*(\theta,\Phi)(\omega):=G^*(\theta,\Phi)(\omega)(v^0_\omega),\]
	for $\omega \in \Omega$, $(\theta, \Phi)\in B_\mathbb C(0,1)\times\mathcal N_0$. 
	
	The proof of the following result can be obtained by repeating the arguments yielding Theorem~\ref{thm:geneigenvec}.
	\begin{proposition}\label{adjoint}
		There exists a neighborhood $\mathcal U'$ of $(0,0)\in B_\mathbb C(0,1)\times \mathcal N_0$, such that:
		\begin{itemize}
			\item The maps $G^*$ and $H^*$ are analytic on $\mathcal U'$ and $\essinf_{\mathcal U'}\|H\|_{L^\infty}\ge \frac{1}{2}$.
			\item The map $F^*:\mathcal U'\subset B_\mathbb C(0,1)\times\mathcal N_0\to\mathcal N_0$, defined by 
			\[F^*(\theta,\Phi):=\frac{G^*(\theta,\Phi)}{H^*(\theta,\Phi)}-\Phi-m,\]
			satisfies $F^*(0,0)=0$, is analytic, and its partial differential w.r.t $\Phi$ at $(0,0)$, $D_2F^*(0,0):\mathcal N_0\to\mathcal N_0$ is invertible.
			\item For any $\theta$ in some neighborhood $U'$ of $0\in\mathbb C$, there exists $\phi^\theta\in\mathcal N$ and $\tilde\lambda^\theta\in L^\infty(\Omega,\mathbb C)$, both analytic in $\theta$ and such that
			\begin{align}\label{eq:geneigenvec*}
				(\L_\omega^\theta)^*\phi^\theta_{\sigma\omega}&=\tilde\lambda_\omega^\theta\phi^\theta_\omega\\
				\tilde\lambda_\omega^\theta&= \phi^\theta_{\sigma\omega}(\L_\omega^\theta v^0_\omega).
			\end{align} 
		\end{itemize}
	\end{proposition}
	We end this section by a remark on the generalized eigenvalues $\tilde\lambda^\theta$ and $\lambda^\theta$:
	\begin{remark}
		A priori, the quantities $\lambda^\theta$, given by \eqref{eq:geneigenvec} and $\tilde\lambda^\theta$, given by \eqref{eq:geneigenvec*} have no reason to coincide. Nevertheless, they do. 
		\par\noindent Indeed, by linearity, we may choose a solution $\phi^\theta$ of \eqref{eq:geneigenvec*} normalized by $\phi^\theta_\omega(v^\theta_\omega)=1$, $\mathds P$-a.e.
		Evaluating both sides of \eqref{eq:geneigenvec*} at $v_\omega^\theta$, one gets, for a.e $\omega\in\Omega$
		\begin{align*}
			\lambda_\omega^\theta\phi^\theta_{\sigma\omega}(v_{\sigma\omega}^\theta)&=\phi^\theta_{\sigma\omega}(\L_\omega^\theta v^\theta_\omega)\\
			&=\tilde\lambda_\omega^\theta\phi^\theta_\omega(v^\theta_\omega),
		\end{align*}
		from which $\lambda^\theta_\omega=\tilde\lambda^\theta_\omega$ $\mathds P$-a.e.
	\end{remark}
	\subsection{Quasi-compactness and regularity of the top Lyapunov exponent of the twisted cocycle}
	In this subsection, we prove the following theorem:
	\begin{theorem}\label{tqc}
		For $\theta$ sufficiently close to $0$, $\L^\theta=(\L_\omega^\theta)_{\omega \in \Omega}$ is a quasi-compact cocycle. Furthermore, the top Oseledets space $Y_1^\theta(\omega)$ is one-dimensional, spanned by $v^\theta_\omega$ (given by \eqref{eq:geneigenvec}). We write
		\begin{equation}\label{eq:twistedsplit}
			BV=Y_1^\theta(\omega)\oplus H^\theta(\omega),
		\end{equation}
		for the associated Oseledets splitting.
	\end{theorem}
	We begin with the following lemma.
	\begin{lemma}\label{wLY}
		There exists a random variable $\tilde C \colon \Omega \to \Omega$ such that $\log \tilde C\in L^1(\Omega)$ and 
		\[
		\|\L_\omega^\theta\phi\|_{BV}\le \tilde C(\omega)\|\phi\|_{BV}, \quad \text{for $\mathbb P$-a.e. $\omega \in \Omega$, $|\theta | \le 1$ and $\phi \in BV$.}
		\]
	\end{lemma}	
	\begin{proof}
		By~\eqref{weakLY}, we have that 
		\[
		\|\L_\omega^\theta\phi\|_{BV}=\|\L_\omega(e^{\theta\psi_\omega}\phi)\|_{BV}\le C(\omega)\|e^{\theta\psi_\omega}\phi\|_{BV}.
		\]
		On the other hand, it is shown in the proof of~\cite[Lemma 3.2]{D} that 
		\[
		\begin{split}
			\|e^{\theta\psi_\omega}\phi\|_{BV} &\le e^{|\Re(\theta)| \cdot \|\psi_\omega\|_{\infty}}\left(1+|\theta|C_{var}\var(\psi_\omega)\right)\|\phi \|_{BV}  \\
			&\le e^{ \|\psi_\omega\|_{\infty}}\left(1+C_{var}\var(\psi_\omega)\right)\|\phi \|_{BV}.
		\end{split}
		\]
		Hence, 
		\[
		\|\L_\omega^\theta\phi\|_{BV} \le C(\omega)\esssup_{\omega \in \Omega} \bigg (e^{ \|\psi_\omega\|_{\infty}}\left(1+C_{var}\var(\psi_\omega)\right) \bigg )\|\phi \|_{BV},
		\]
		for $\mathbb P$-a.e. $\omega \in \Omega$, $|\theta | \le 1$ and $\phi \in BV$.  In order to complete the proof, it remains to recall that $\log C\in L^1(\Omega)$ and to note that
		\[
		\esssup_{\omega \in \Omega} \bigg (e^{ \|\psi_\omega\|_{\infty}}\left(1+C_{var}\var(\psi_\omega)\right) \bigg  ) <+\infty, 
		\]
		which follows from~\eqref{eq:condobservable} (see Remark~\ref{r0}).
	\end{proof}
	
	\begin{remark}\label{228}
		Observe that in the proof of Lemma~\ref{wLY} we actually showed that $\|\L_\omega^\theta\|_{BV}$ was bounded a positive, log-integrable random variable $C_\theta$, going to $C$ a.s as $\theta\to 0$, and dominated by a (still log-integrable) $\tilde C$ of the form $cC$ for some constant $c>0$.
	\end{remark}
	
	It follows from Lemma~\ref{wLY} that the top Lyapunov exponent of the twisted coycle $\L^\theta=(\L_\omega^\theta)_{\omega \in \Omega}$ exists for each $|\theta | \le 1$. We will denote it by $\Lambda (\theta)$. Following~\cite{D}, we introduce an auxiliary quantity. 
	More precisely, for $\theta \in U$ let 
	\[
	\hat{\Lambda} (\theta)=\int_{\Omega}\log|\lambda_\omega^\theta | \, d\mathbb P(\omega),
	\]
	where $\lambda_\omega^\theta$ is given by \eqref{eq:geneigenvec}.
	For $\theta \in U$ and $\omega \in \Omega$, set
	\[
	\L_\omega^{\theta, n}:=\L_{\sigma^{n-1} \omega}^\theta \circ \ldots \circ \L_\omega^\theta. 
	\]
	We recall (see~\cite[Lemma 3.3]{D}) that 
	\begin{equation}\label{tbs}
		\L_\omega^{\theta, n}(\phi)=\L_\omega^n (e^{ \theta S_n \psi(\omega, \cdot) } \phi), 
	\end{equation}
	where
	\begin{equation}\label{bs}
		S_n \psi (\omega, \cdot)=\sum_{i=0}^{n-1} \psi(\sigma^i \omega, T_\omega^i(\cdot)),
	\end{equation}
	and 
	\[
	T_\omega^i=T_{\sigma^{i-1} \omega} \circ \ldots \circ T_\omega. 
	\]
	The following is a version of~\cite[Lemma 3.8.]{D} in our setting. 
	\begin{lemma}
		For $\theta \in U$, we have that $\Lambda (\theta) \ge \hat{\Lambda} (\theta)$.
	\end{lemma}
	
	\begin{proof}
		By~\eqref{LN}, \eqref{an1}, \eqref{eq:geneigenvec}  and recalling that $K$ is tempered, we have that 
		\[
		\begin{split}
			\Lambda (\theta) &\ge \limsup_{n\to \infty} \frac 1n \log \| \L_\omega^{\theta, n} v_\omega^\theta \|_{BV}    \\
			&\ge \limsup_{n\to \infty} \frac 1 n \log \bigg (K(\sigma^n \omega)^{-1} \| \L_\omega^{\theta, n} v_\omega^\theta \|_{\sigma^n \omega} \bigg ) \\
			&=\limsup_{n\to \infty} \frac 1 n \log \| \L_\omega^{\theta, n} v_\omega^\theta \|_{\sigma^n \omega}  \\
			&\ge \limsup_{n\to \infty} \frac 1 n \log \| \L_\omega^{\theta, n} v_\omega^\theta \|_1 \\
			&=\limsup_{n\to \infty} \frac 1n \sum_{k=0}^{n-1} \log |\lambda_{\sigma^k \omega}^\theta |+\limsup_{n\to \infty} \frac 1 n \log \|v_{\sigma^n \omega}^\theta \|_1 \\
			&\ge \limsup_{n\to \infty} \frac 1n \sum_{k=0}^{n-1} \log |\lambda_{\sigma^k \omega}^\theta | ,
		\end{split}
		\]
		for $\mathbb P$-a.e. $\omega \in \Omega$, where in the last step we have used that 
		\[
		1=\bigg |\int_X v_\omega^\theta \, dm\bigg | \le \|v_\omega^\theta \|_1.
		\]
		It remains to observe that Birkhoff ergodic theorem implies that 
		\[
		\lim_{n\to \infty} \frac 1n \sum_{k=0}^{n-1} \log |\lambda_{\sigma^k \omega}^\theta |=\hat{\Lambda} (\theta),
		\]
		for $\mathbb P$-a.e. $\omega \in \Omega$.
	\end{proof}
	
	\begin{lemma}\label{difflambda}
		We have that $\hat\Lambda$ is real-analytic and harmonic on a neighborhood of 0. Moreover, $\hat{\Lambda}'(0)=0$.
	\end{lemma}
	
	\begin{proof}
		The first part follows by noting that, by Lemma \ref{lemma:derivatives} and the dominated convergence theorem, $\hat\Lambda$ is the real-part of complex-analytic function.
		\par\noindent The second part is an easy consequence of~\eqref{center} (see~\cite[Lemma 3.11]{D}).
	\end{proof}
	
	We are now able to prove the main theorem of this section:
	\begin{proof}[Proof of Theorem \ref{tqc}]
		Throughout this proof $D>0$ will denote a generic constant (independent on $\omega$, $\theta$ and $\phi$) that can change from line to line. 
		We begin by noting that~\eqref{strongLY} implies that 
		\begin{align*}
			\|\L_{\omega}^{\theta, N} \phi\|_{BV}&\le \|\L_{\omega}^N\phi\|_{BV}+\|(\L_{\omega}^{\theta, N}-\L_\omega^N)\phi\|_{BV}\\
			&\le \alpha^N(\omega)\|\phi\|_{BV}+(K^N(\omega)+1)\|\phi\|_1+\|(\L_{\omega}^{\theta, N}-\L_\omega^N)\phi\|_{BV}.
		\end{align*}
		By using the same arguments as in the proof of Lemma~\ref{24} we have that 
		\begin{align*}
			\|(\L_{\omega}^\theta -\L_\omega)\phi\|_{BV}&=\|\L_\omega\left(e^{\theta\psi_\omega}-1\right)\phi\|_{BV}\\
			&\le C(\omega)\|\left(e^{\theta\psi_\omega}-1\right)\phi\|_{BV}\\
			&\le C(\omega)C_{var}|\theta|e^{\|\psi_\omega\|_\infty}\left(\|\psi_\omega\|_\infty+\var(\psi_\omega)\right) \|\phi \|_{BV}.
		\end{align*}
		Moreover, we have 
		\[(\L_{\omega}^{\theta, N}-\L_\omega^N)\phi=\sum_{j=0}^{N-1}\L_{\sigma^{N-j}\omega}^{\theta, j}\left(\L_{\sigma^{N-1-j}\omega}^{\theta}-\L_{\sigma^{N-1-j}\omega}\right)\L_\omega^{N-1-j}\phi.\]
		By combining the above facts together with~\eqref{weakLY}, Lemma~\ref{wLY} (see also Remark~\ref{228}) and~\eqref{b}, we obtain that 
		\[
		\begin{split}
			& \| \L_{\sigma^{N-j}\omega}^{\theta, j}\left(\L_{\sigma^{N-1-j}\omega}^{\theta}-\L_{\sigma^{N-1-j}\omega}\right)\L_\omega^{N-1-j}\phi \|_{BV} \\
			&\le D |\theta| \prod_{k=0}^{N-1}C(\sigma^k\omega)\|\phi\|_{BV}e^{\|\psi_{\sigma^{N-1-j}\omega}\|_\infty}\left(\|\psi_{\sigma^{N-1-j}\omega}\|_\infty+\var(\psi_{\sigma^{N-1-j}\omega})\right)\\
			&\le D|\theta| \prod_{k=0}^{N-1}C(\sigma^k\omega)\|\phi\|_{BV},
		\end{split}
		\]
		for $\mathbb P$-a.e. $\omega \in \Omega$, $|\theta | \le 1$, $\phi \in BV$ and $j \in \{0, \ldots, N-1\}$. Hence, 
		\[
		\|(\L_{\omega}^{\theta, N}-\L_\omega^N)\phi  \|_{BV} \le  D|\theta| \prod_{k=0}^{N-1}C(\sigma^k\omega)\|\phi\|_{BV},
		\]
		and thus
		\begin{equation}\label{34}
			\|\L_{\omega}^{\theta, N} \phi\|_{BV} \le \bigg (\alpha^N(\omega)+D|\theta| \prod_{k=0}^{N-1}C(\sigma^k\omega) \bigg )\| \phi \|_{BV}+(K^N(\omega)+1)\|\phi\|_1,
		\end{equation}
		for $\mathbb P$-a.e. $\omega \in \Omega$, $|\theta | \le 1$ and $\phi \in BV$. From Lemma~\ref{QC}, Lemma~\ref{wLY} and~\eqref{34}, by arguing as in the proof of~\cite[Theorem 3.12]{D}, we get that the twisted transfer operator cocycle is quasi-compact. 
		\par\noindent The fact that the top Oseledets space is one-dimensional now follows from \cite[Lemma A.3]{D}, arguing as in the second part of the proof of \cite[Theorem 3.12]{D}.
	\end{proof}
	
	\begin{proposition}
		For $\theta$ sufficiently close to $0$, $\Lambda (\theta)=\hat{\Lambda} (\theta)$. In particular, $\Lambda$ is real-analytic and harmonic on  a neighborhood of $0$.
	\end{proposition}
	
	\begin{proof}
		This follows from Lemma~\ref{difflambda} and Theorem~\ref{tqc} by arguing exactly  as in the proof of~\cite[Corollary 3.14.]{D}.
	\end{proof}
	
	As a consequence of Theorems \ref{tqc} and \ref{thm:dualcocycle}, one has:
	\begin{cor}
		The dual twisted transfer operator cocycle $(\L^\theta)^*$ is quasi-compact, and its top Oseledets space, $(Y_1^\theta(\omega))^*$ is one-dimensional, spanned by $\phi_\omega^\theta$ given by \eqref{eq:geneigenvec*}.
	\end{cor}
	
	\section{Limit theorems: new results with ``old'' proofs}
	
	\subsection{Variance}
	By $\mathbb E_\omega(\phi)$ we will denote $\int_X \phi \, d\mu_\omega$, where $\mu_\omega$, $\omega \in \Omega$ are as in the statement of Theorem~\ref{7}. Moreover, $\mu$ will denote the measure on $\Omega \times X$ given by
	\[
	\mu(A\times B)=\int_A\mu_\omega(B)\, d\mathbb P(\omega), \quad \text{for $A\in \mathcal F$ and $B\in \mathcal G$.}
	\]
	We note that $\mu$ is invariant and ergodic for the skew-product transformation $\tau \colon \Omega \times X \to \Omega \times X$ given by
	\[
	\tau(\omega, x)=(\sigma \omega, T_\omega(x)), \quad (\omega, x) \in \Omega \times X.
	\]
	\begin{lemma}\label{var}
		There exists $\Sigma^2\ge 0$ such that
		\begin{equation}\label{variance}
			\lim_{n\to \infty} \frac 1 n \mathbb E_\omega \bigg{(}\sum_{k=0}^{n-1}  \psi_{\sigma^k \omega} \circ T_\omega^k \bigg{)}^2=\Sigma^2, \quad \text{for a.e. $\omega \in \Omega$.}
		\end{equation}
	\end{lemma}
	
	\begin{proof}
		By using~\eqref{b} (which is a consequence of~\eqref{eq:condobservable}) and arguing as in the proof of~\cite[Lemma 12]{D1}, we find that 
		\[
		\mathbb E_\omega \bigg{(}\sum_{k=0}^{n-1}  \psi_{\sigma^k \omega} \circ T_\omega^k \bigg{)}^2=\sum_{k=0}^{n-1} \mathbb E_\omega ( \psi_{\sigma^k \omega}^2\circ T_\omega^k )
		+2\sum_{i=0}^{n-1}\sum_{j=i+1}^{n-1}\mathbb E_{\sigma^i \omega}( \psi_{\sigma^i \omega}( \psi_{\sigma^j \omega}\circ T_{\sigma^i \omega}^{j-i}))
		\]
		and 
		\[
		\lim_{n\to \infty} \frac 1 n  \sum_{k=0}^{n-1}  \mathbb E_\omega ( \psi_{\sigma^k \omega}^2\circ T_\omega^k )=\int_{\Omega \times X} \psi(\omega, x)^2\, d\mu (\omega, x),
		\]
		for $\mathbb P$-a.e. $\omega \in \Omega$. Set 
		\[
		\Psi(\omega)=\sum_{n=1}^\infty \int_X   \psi(\omega, x) \psi(\tau^n (\omega, x))\, d\mu_\omega(x)=\sum_{n=1}^\infty \int_X\mathcal L_\omega^n(\psi_\omega v_\omega^0) \psi_{\sigma^n \omega} \, dm.
		\]
		By~\eqref{436} and noting that  $\max\{ D(\omega), \|v_\omega^0 \|_{BV} \} \le K(\omega)$ (see the proof of Proposition~\ref{AN}), we have that 
		\[
		\begin{split}
			\lvert \Psi(\omega)\rvert  &\le \sum_{n=1}^\infty \| \mathcal L_\omega^n(\psi_\omega v_\omega^0) \|_{BV} \cdot \| \psi_{\sigma^n \omega} \|_{BV}\\
			&\le  K(\omega) \esssup_{\omega \in \Omega} \|\psi_\omega \|_{BV}\sum_{n=1}^\infty e^{-\lambda n} \| \psi_\omega v_\omega^0 \|_{BV} \\
			&\le \frac{C_{var} K(\omega)^2\|\psi_\omega \|_{BV}\cdot  \esssup_{\omega \in \Omega} \|\psi_\omega \|_{BV}}{1-e^{-\lambda}},
		\end{split}
		\]
		for $\mathbb P$-a.e. $\omega \in \Omega$. Using~\eqref{eq:condobservable}, we conclude that $\Psi \in L^\infty (\Omega)$ and thus it follows again from Birkhoff ergodic theorem that, 	for $\mathbb P$  a.e. $\omega \in \Omega$,
		\begin{align}\label{DD}
			\notag\lim_{n\to \infty}\frac{1}{n} \sum_{i=0}^{n-1}\Psi(\sigma^i \omega)&=\int_\Omega \Psi(\omega)\, d\mathbb P(\omega)\\
			&=\sum_{n=1}^\infty \int_{ \Omega \times X}\psi (\omega, x)\psi (\tau^n(\omega, x))\, d\mu (\omega, x).
		\end{align}
		Moreover, by arguing as in the proof of~\cite[Lemma 12]{D1}, we have that 
		\[
		\begin{split}
			&  \bigg{\lvert} \sum_{i=0}^{n-1}\sum_{j=i+1}^{n-1}\mathbb E_{\sigma^i \omega}( \psi_{\sigma^i \omega}( \psi_{\sigma^j \omega}\circ T_{\sigma^i \omega}^{j-i}))-\sum_{i=0}^{n-1}\Psi(\sigma^i \omega) \bigg{\rvert} \\
			& \le \sum_{i=0}^{n-1}\sum_{k=n-i}^\infty  \bigg{\lvert}\int_X \mathcal L_{\sigma^i \omega}^k( \psi_{\sigma^i \omega} v^0_{ \sigma^i \omega})  \psi_{\sigma^{k+i} \omega}\, dm \bigg{\rvert}  \\
			&\le \esssup_{\omega \in \Omega} \|\psi_\omega \|_{BV} \cdot \esssup_{\omega \in \Omega} (K(\omega)^2\|\psi_\omega \|_{BV}) \sum_{i=0}^{n-1}\sum_{k=n-i}^\infty e^{-\lambda k}, 
		\end{split} 
		\]
		which implies that 
		\begin{equation}\label{cv}
			\lim_{n\to \infty} \frac 1 n \bigg{(}\sum_{i=0}^{n-1}\sum_{j=i+1}^{n-1}\mathbb E_{\sigma^i \omega}(\psi_{\sigma^i \omega}( \psi_{\sigma^j \omega}\circ T_{\sigma^i \omega}^{j-i}))-\sum_{i=0}^{n-1}\Psi(\sigma^i \omega)\bigg{)}=0,
		\end{equation}
		It follows from~\eqref{DD} and~\eqref{cv} that
		\begin{small}
			\[
			\lim_{n \to \infty} \frac 1 n \sum_{i=0}^{n-1}\sum_{j=i+1}^{n-1}\mathbb E_{\sigma^i \omega}(\psi_{\sigma^i \omega}( \psi_{\sigma^j \omega}\circ T_{\sigma^i \omega}^{j-i}))=\sum_{n=1}^\infty \int_{ \Omega \times X} \psi (\omega, x)
			\psi (\tau^n(\omega, x))\, d\mu (\omega, x)
			\]
		\end{small}
		for a.e. $\omega \in \Omega$ and therefore~\eqref{variance} holds with
		\begin{equation}\label{var2}
			\Sigma^2=\int_{\Omega \times X}\psi(\omega, x)^2\, d\mu (\omega, x)+2\sum_{n=1}^\infty \int_{ \Omega \times X}\psi (\omega, x)
			\psi (\tau^n(\omega, x))\, d\mu (\omega, x).
		\end{equation}
		Finally, we note that it follows readily from~\eqref{variance} that  $\Sigma^2\ge 0$ and the proof of the lemma is completed.
	\end{proof}
	
	The proof of the following result can be done by arguing exactly as in the proof of~\cite[Proposition 3]{D1}.
	\begin{proposition}
		We have that $\Sigma^2=0$ if and only if there exist $\phi \in L^2(\Omega \times X, \mu)$ such that 
		\[
		\psi=\phi-\phi \circ \tau.
		\]
	\end{proposition}
	
	\begin{cor}\label{cor:2difflambda}
		One has $\Lambda''(0)=\Sigma^2$. In particular, if $\Sigma^2>0$, $\Lambda$ is strictly convex on some neighborhood of $0$. 
	\end{cor}
	\begin{proof}
		The value of $\Lambda''(0)$ is easily obtained by Lemma \ref{lemma:derivatives} and \eqref{var2} (see also \cite[Section 3.6]{D} for a similar argument), and the rest is an elementary fact of real analysis.
	\end{proof}
	\subsection{Proof of Theorem \ref{thm:CLT}}  
	In this section, we prove that if $\Sigma^2>0$ then for $\mathbb P$-a.e $\omega\in\Omega$, the process $(\psi_{\sigma^n\omega}\circ T_\omega^n)_{n\ge 0}$ satisfies the central limit theorem.  More precisely, we show that for $\mathbb P$-a.e. $\omega \in \Omega$ and 
	every bounded and continuous $\phi \colon \mathbb R \to \mathbb R$ we have that 
	\[
	\lim_{n\to \infty}\int_X \phi \bigg (\frac{S_n \psi(\omega, x)}{\sqrt n}\bigg )\, d\mu_\omega(x)=\int_{\mathbb R} \phi \, d\mathcal N(0, \Sigma^2),
	\]
	where $\mathcal N(0, \Sigma^2)$ denotes the normal distribution (with parameters $0$ and $\Sigma^2$) and $S_n \psi$ is given by~\eqref{bs}.
	
	In order to establish the above claim, we essentially follow \cite[Proof of Thm B.]{D} and the classical method given by Levy's continuity theorem:
	\\Writing $t_n=t/\sqrt n$,  we have that 
	\begin{equation}\label{deco}
		\begin{split}
			\int_X e^{it_nS_n \psi (\omega, \cdot) } v^0_\omega dm&= \int_X \L^{it_n, n}_{\omega}v^0_\omega dm\\
			&=\int_X\L_{\omega}^{it_n, n}(v^0_\omega-v^{it_n}_{\omega}) dm + \int_X \L_{\omega}^{it_n, n}v^{it_n}_{\omega} dm\\
			&= \int_X \L_{\omega}^{it_n, n}(v^0_\omega-v^{it_n}_{\omega}) dm +\prod_{j=0}^{n-1}\lambda^{it_n}_{\sigma^j\omega}\int_X v^{it_n}_{\omega} dm\\
			&= \int_X\L_{\omega}^{it_n, n}(v^0_\omega-v^{it_n}_{\omega}) dm +\prod_{j=0}^{n-1}\lambda^{it_n}_{\sigma^j\omega}
		\end{split}
	\end{equation}
	Observe that~\eqref{an1} and~\eqref{tbs}  imply that 
	\[
	\begin{split}
		\bigg |\int_X\L_{\omega}^{it_n, n}(v^0_\omega-v^{it_n}_{\omega}) dm \bigg | &= \bigg | \int_X \L_\omega^n (e^{it_n S_n \psi (\omega, \cdot)}(v^0_\omega-v^{it_n}_{\omega})) dm \bigg | \\
		&=\bigg | \int_X e^{it_n S_n \psi (\omega, \cdot)}(v^0_\omega-v^{it_n}_{\omega})dm  \bigg | \\
		&\le \|v^0_\omega-v^{it_n}_{\omega} \|_1 \\
		&\le \|v^0_\omega-v^{it_n}_{\omega} \|_\omega \\
		&\le \|v^0-v^{it_n} \|_{\S},
	\end{split}
	\]
	for $\mathbb P$-a.e. $\omega \in \Omega$ and $n\in \N$. Thus, from Theorem~\ref{thm:geneigenvec} we conclude that 
	\[
	\lim_{n\to \infty} \int_X\L_{\omega}^{it_n, n}(v^0_\omega-v^{it_n}_{\omega}) dm =0, \quad \text{for $\mathbb P$-a.e. $\omega \in \Omega$.}
	\]
	Hence, we have to show that 
	\[\lim_{n\to\infty}\prod_{j=0}^{n-1}\lambda^{it_n}_{\sigma^j\omega}=e^{-t^2\Sigma^2/2} \quad \text{for $\mathbb P$-a.e. $\omega \in \Omega$,} \]
	which is equivalent to
	\[\lim_{n\to\infty}\sum_{j=0}^{n-1}\log \lambda^{it_n}_{\sigma^j\omega}=-\frac{t^2\Sigma^2}{2} \quad \text{for $\mathbb P$-a.e. $\omega \in \Omega$.} \]
	Let us consider the map $H$, defined from a neighborhood of $0$ in $\mathbb C$ to $L^\infty(\Omega)$, by $H(\theta)(\omega)=\log \lambda_\omega^\theta$, $\omega \in \Omega$. Up to shrinking its domain, this map is analytic, as the composition of two analytic maps, and satisfies by Lemma \ref{lemma:derivatives} $H(0)(\omega)=0$, $H'(0)(\omega)=0$ and
	\begin{equation*}
		H''(0)(\omega)=\frac{d^2\lambda^\theta_\omega}{d\theta^2}(0)=\int_X \psi_\omega^2 v^0_\omega dm+ 2\sum_{j=1}^\infty\int_X \psi_\omega\left[\L^{(j)}_{\sigma^{-j}\omega}(\psi_{\sigma^{-j}\omega}v^0_{\sigma^{-j}\omega})\right]dm.
	\end{equation*}
	By Taylor's formula at order two, one has
	\begin{equation*}
		H(it_n)(\omega)= -\frac{t^2}{2n}H''(0)(\omega)+R(it_n)(\omega),
	\end{equation*}
	with $R$ the remainder of the series. One then has 
	\begin{equation*}
		\sum_{j=0}^{n-1}H(it_n)(\sigma^j\omega)=- \frac{t^2}{2}\frac{1}{n}\sum_{j=0}^{n-1}H''(0)(\sigma^j\omega)+\sum_{j=0}^{n-1}R(it_n)(\sigma^j\omega).
	\end{equation*}
	By Birkhoff ergodic theorem, as $n\to\infty$,
	\[\frac{1}{n}\sum_{j=0}^{n-1}H''(0)(\sigma^j\omega)\to \int_\Omega H''(0)(\omega)d\mathbb P(\omega)=\Sigma^2, \quad \text{for $\mathbb P$-a.e. $\omega \in  \Omega$.}\]
	For the remainder term, we notice that  one may write $R(\theta)=\theta^2\tilde R(\theta)$, with $\|\tilde R(\theta)\|_{L^\infty}\to 0$ when $\theta\to 0$.  In particular, for any $\epsilon>0$ and $t\in\R\backslash\{0\}$, one may find a $\delta>0$, such that $\|\tilde R(\theta)\|_{L^\infty}\le \frac{\epsilon}{t^2}$ if $|\theta|\le \delta$. There is also a $n_1$, such that for $n\ge n_1$,
	$|it_n|\le \delta$. Putting this together yields, for $n\ge n_1$,
	\[\left|\sum_{j=0}^{n-1}R(it_n)(\sigma^j\omega)\right|\le \frac{t^2}{n}\sum_{j=0}^{n-1}\frac{\epsilon}{t^2}\le \epsilon,\]
	i.e. $\lim_{n\to\infty}\sum_{j=0}^{n-1}R_{\sigma^j\omega}(it_n)=0$, for $\mathbb P$-a.e. $\omega \in \Omega$. The announced result follows. 
	
	\begin{remark}
		As mentioned, our proof of the  quenched C.L.T. is similar to the proof of~\cite[Theorem B]{D}. However, we note that our initial step \eqref{deco} is different from that in~\cite{D}. In particular, we don't need the version of~\cite[Lemma 4.4]{D} in our setting. 
	\end{remark}
	
	\subsection{Proof of Theorem \ref{thm:LDT}}
	In this section, we prove a (quenched) large deviations estimate for the process $(S_n\psi(\omega, \cdot))_{n\ge 0}$. We need the following classical  result.
	\begin{theorem}\label{GET}(G\"artner-Ellis \cite{HennionHerve})
		For  $n\in \N$, let $\mathbb P_n$ be a probability measure on a measurable space $(Y, \mathcal T)$ and let $\mathbb E_n$ denote the corresponding expectation operator. Furthermore,
		let $S_n$ be a real random variable on $(\Omega, \mathcal T)$ and assume that on some interval $[-\theta_{+}, \theta_{+}]$, $\theta_{+}>0$, we have
		\begin{equation}\label{GETE}
			\lim_{n\to \infty} \frac 1n \log \mathbb E_n (e^{\theta S_n})=\psi (\theta),
		\end{equation}
		where $\psi$ is a strictly convex continuously differentiable function satisfying $\psi'(0)=0$. Then, there exists $\epsilon_+>0$
		such that the function $c$ defined by
		\begin{equation}\label{GETE2}
			c(\epsilon)=\sup_{\lvert \theta \rvert \le \theta_+ }\{ \theta \epsilon-\psi (\theta) \}
		\end{equation}
		is
		nonnegative, continuous, strictly convex  on $[-\epsilon_+, \epsilon_+]$,
		vanishing only at $0$ and  such that
		\[
		\lim_{n\to \infty} \frac 1 n \log \mathbb P_n (S_n >n\epsilon)=-c(\epsilon), \quad \text{for every $\epsilon \in (0, \epsilon_+)$.}
		\]
	\end{theorem}
	
	We need the following lemma, whose statement and proof are similar to \cite[Lemma 4.2]{D}. As such, we will omit the proof.
	\begin{lemma}\label{lemma:growthrate}
		Let $\theta\in\mathbb C$ be close enough to $0$. For any $f\in BV$ such that $\phi^\theta_\omega(f)\not=0$, one has
		\begin{equation}
			\lim_{n\to\infty}\frac{1}{n}\log\left|\int_X e^{\theta S_n\psi (\omega, \cdot) }f dm\right|= \Lambda(\theta).
		\end{equation}
	\end{lemma}
	\begin{proof}[Proof of quenched large deviations]
		We start by remarking that Lemma \ref{lemma:growthrate} implies that 
		\[\lim_{n\to\infty}\frac{1}{n}\log\left|\int_X e^{\theta S_n\psi(\omega, \cdot)}d\mu_\omega\right|=\Lambda(\theta).\]
		Indeed, since we may write $\int_X e^{\theta S_n\psi(\omega, \cdot) }d\mu_\omega=\int_X e^{\theta S_n\psi(\omega, \cdot)} v_\omega^0 dm$, it suffices to show that $\phi^\theta_\omega(v^0_\omega)\not=0$. But this easily follows from $1=\phi^0_\omega(v^0_\omega)=\int_X v_\omega^0 dm$ and analyticity of $\theta\mapsto\phi_\omega^\theta\in\mathcal N$ in some neighborhood of $0$.
		\par\noindent Hence, since $\Lambda$ is convex in some (small enough) real neighborhood of $0$ (see Corollary \ref{cor:2difflambda}), we may apply Theorem~\ref{GET}
		for each fixed $\omega\in\Omega$ in some full-measure subset, with $\mathds P_n=\mu_\omega$, $S_n=S_n\psi_\omega$ and $\psi(\theta)=\Lambda(\theta)$.
	\end{proof}
	
	\subsection{Proof of Theorem \ref{thm:LCLT}}
	We first establish the following version of~\cite[Lemma 4.7]{DH} in our setting. 
	\begin{lemma}\label{444}
		There exist $C>0$ and $0<r<1$ such that for every $\theta \in \mathbb{C}$ sufficiently close to 0, every $n\in \N$ and $\mathbb P$-a.e. $\omega \in \Omega$, we have
		\[
		\Big| \int_X\mathcal L_\omega^{\theta, n}(v_\omega^0 -\phi_\omega^{\theta}(v_\omega^0) v_{\omega}^{\theta})\, dm \Big|
		\leq Cr^n |\theta |.
		\]
	\end{lemma}
	
	\begin{proof}
		For $\theta$ near 0 and $\omega \in \Omega$, let
		\[
		\mathcal Q_\omega^{\theta} h:=  \mathcal L_\omega^{\theta}(h- \phi_\omega^\theta (h) v_{\omega}^{\theta}), \quad h\in BV.
		\]
		Then, 
		\begin{equation}\label{iter}
			\mathcal Q_\omega^{\theta, n} h =  \mathcal L_\omega^{\theta, n}(h - \phi_\omega^\theta (h) v_{\omega}^{\theta}), \quad \text{for $\omega \in \Omega$, $n\in \N$ and $h\in BV$.}
		\end{equation}
		Observe that~\eqref{an2} gives that
		\begin{equation}\label{deC}
			\| \mathcal Q_\omega^0h \|_{\sigma \omega} \le e^{-\lambda} \|h \|_\omega, \quad \text{for $\mathbb P$-a.e. $\omega \in \Omega$ and $h\in BV$.}
		\end{equation}
		Take $r>0$ such that $e^{-\lambda} <r<1$. It follows from Theorem~\ref{thm:geneigenvec}, Proposition~\ref{adjoint} and~\eqref{deC} that for $\theta$ sufficiently close to $0$, 
		\begin{equation}\label{deC1}
			\|\mathcal Q_\omega^\theta h\|_{\sigma \omega} \le r \|h \|_\omega, \quad \text{for $\mathbb P$-a.e. $\omega \in \Omega$ and $h\in BV$.}
		\end{equation}
		By iterating~\eqref{deC1} and using~\eqref{iter}, we obtain that 
		\begin{equation}\label{deC2}
			\|\mathcal L_\omega^{\theta, n}(h - \phi_\omega^\theta (h) v_{\omega}^{\theta})\|_{\sigma^n \omega} \le r^n \|h \|_\omega \quad \text{for $\mathbb P$-a.e. $\omega \in \Omega$, $n\in \N$ and $h\in BV$,}
		\end{equation}
		whenever $\theta$ is sufficiently close to $0$. Now it follows readily from~\eqref{an1}, \eqref{an4} and~\eqref{deC2} that there exists $C>0$ such that for $\theta$ sufficiently close to $0$, $\mathbb P$-a.e. $\omega \in \Omega$ and $n\in \N$,
		\begin{equation}\label{2016}
			\Big| \int_X\mathcal L_\omega^{\theta, n}(v_\omega^0 -\phi_\omega^{\theta}(v_\omega^0) v_{\omega}^{\theta})\, dm \Big| \le Cr^n.
		\end{equation}
		The conclusion of the lemma now follows from the Cauchy integral formula and a simple observation that the term $\int_X\mathcal L_\omega^{\theta, n}(v_\omega^0 -\phi_\omega^{\theta}(v_\omega^0) v_{\omega}^{\theta})\, dm $ vanishes at $\theta =0$.
	\end{proof}
	
	\begin{remark}
		We note that for the purpose of establishing  the quenched local C.L.T~\eqref{2016} is sufficient. However, the finer conclusion given by Lemma~\ref{444} is needed for the Berry-Essen estimates (see~\cite[Section 4.4]{DH}).
	\end{remark}
	
	We now establish the quenched local C.L.T under the following \emph{aperiodicity} assumption. Namely, we require that for $\mathbb P$-a.e. $\omega \in \Omega$ and every compact interval $J\subset \mathbb R \setminus \{0\}$, there exist $C=C(\omega)>0$ and $\rho \in (0,1)$ such that 
	\begin{equation}\label{ap}
		\| \L_\omega^{it, n} \|_{BV}  \le C\rho^n, \quad \text{for $t\in J$ and $n\ge 0$.}
	\end{equation}
	Following~\cite{D, RE83}, it is sufficient to show that for $h\in L^1(\R)$ whose Fourier transform $\hat h$ has a compact support, 
	\begin{equation}\label{CLTL1}
		\sup_{s\in \R}\bigg{\rvert} \frac{\Sigma}{2\pi}\int_{\R} e^{\frac{its}{\sqrt n}}\hat h\left(\frac{t}{\sqrt n}\right)\int \mathcal L_{\omega}^{\frac{it}{\sqrt n}, n}v_\omega^0\, dm \, dt -
		\frac{\hat h(0)\Sigma}{2\pi} \int_{\R} e^{\frac{its}{\sqrt n}} \cdot e^{-\frac{\Sigma^2 t^2}{2}}\, dt \bigg{\rvert} \to 0,
	\end{equation}
	when $n\to \infty$, for $\mathbb P$-a.e. $\omega \in \Omega$. 
	
	Choose $\delta >0$ such that the support of $\hat h$ is contained in $[-\delta, \delta]$. For any $\tilde \delta \in (0, \delta)$, we have that 
	\begin{align*}
		&\frac{\Sigma}{2\pi}\int_{\R} e^{\frac{its}{\sqrt n}}\hat h\left(\frac{t}{\sqrt n}\right)\int_X \mathcal L_{\omega}^{\frac{it}{\sqrt n}, n}v_\omega^0\, dm \, dt -
		\frac{\hat h(0)\Sigma}{2\pi} \int_{\R} e^{\frac{its}{\sqrt n}} \cdot e^{-\frac{\Sigma^2 t^2}{2}}\, dt \displaybreak[0] \\
		&=\frac{\Sigma}{2\pi} \int_{\lvert t\rvert < \tilde \delta \sqrt n} e^{\frac{its}{\sqrt n}} \Big(\hat h\left(\frac{t}{\sqrt n}\right)\prod_{j=0}^{n-1}\lambda_{\sigma^j \omega}^{\frac{it}{\sqrt n}}-\hat h(0)e^{-\frac{\Sigma^2 t^2}{2}} \Big)\, dt \displaybreak[0] \\
		&\phantom{=}+\frac{\Sigma}{2\pi}\int_{\lvert t\rvert < \tilde \delta \sqrt n}e^{\frac{its}{\sqrt n}}\hat h\left(\frac{t}{\sqrt n}\right)
		\int_X
		\prod_{j=0}^{n-1}\lambda_{\sigma^j \omega}^{\frac{it}{\sqrt n}} \Big( \phi_\omega^{\frac{it}{\sqrt n}}( v_\omega^0 ) v_{\sigma^{n}\omega}^{\frac{it}{\sqrt n}}-1 \Big)\,dm \, dt \displaybreak[0] \\
		&\phantom{=}+\frac{\Sigma \sqrt{n}}{2\pi}\int_{\lvert t\rvert <\tilde \delta}e^{its}\hat h(t)\int_X \mathcal L_{\omega}^{it, (n)} (v_\omega^0 - \phi_\omega^{it}( v_\omega^0 ) v_{\omega}^{it}) \, dm\, dt \displaybreak[0] \\
		&\phantom{=}+\frac{\Sigma \sqrt{n}}{2\pi}\int_{\tilde \delta \le \lvert t\rvert < \delta}e^{its}\hat h(t)\int_X \mathcal L_\omega^{it, (n)}v_\omega^0\, dm\, dt \displaybreak[0] \\
		&\phantom{=}-\frac{\Sigma}{2\pi}\hat h(0) \int_{\lvert t\rvert \ge \tilde \delta \sqrt n}e^{\frac{its}{\sqrt n}} \cdot e^{-\frac{\Sigma^2 t^2}{2}}\, dt=: (I)+(II)+(III)+(IV)+(V)
	\end{align*}
	By repeating the arguments in the proof of~\cite[Theorem C]{D}, we show that the terms $(I)-(V)$ converge to $0$ uniformly in $s$, when $n\to \infty$. In particular, the aperiodicity condition~\eqref{ap} will take care of $(IV)$,  Lemma~\ref{444} enables us to handle $(III)$,
	while the fact that $(V)$ converges to $0$ uniformly in $s$ follows easily from the dominated convergence theorem. In order to handle the terms $(I)$ and $(II)$, one needs the following result whose statement and the proof is the same as~\cite[Lemma 4.6]{D}.
	\begin{lemma}\label{ProdLam}
		For $\tilde \delta >0$ sufficiently small, there exists $n_0\in \N$ such that for $\mathbb P$-a.e. $\omega \in \Omega$, $n\ge n_0$ and $t$ such that $\lvert t\rvert < \tilde \delta \sqrt n$,
		\[
		\bigg{\lvert} \prod_{j=0}^{n-1} \lambda_{\sigma^j \omega}^{\frac{it}{\sqrt n}} \bigg{\rvert} \le e^{-\frac{t^2 \Sigma^2}{8}}.
		\]
		
	\end{lemma}
	
	\begin{remark}
		Let us comment a bit on the aperiodicity assumption:
		\begin{itemize}
			\item We refer to~\cite[Section 4.3.2]{D} for certain equivalent formulations of the condition~\eqref{ap} (that readily apply to our setting).
			\item One can also formulate and proof the periodic version of the local central limit theorem just as in~\cite[Section 4.4]{D}. 
		\end{itemize}
	\end{remark}
	
	\section{Acknowledgments}
	We would like to thank Yeor Hafouta for several useful comments. 
	
	DD was supported in part by Croatian Science Foundation under the project
	IP-2019-04-1239 and by the University of Rijeka under the projects uniri-prirod-18-9 and uniri-pr-prirod-19-16.
	
	JS was supported by the European Research Council (ERC) under the European Union's Horizon 2020 research and innovation programme (grant agreement No 787304).

	\begin{footnotesize} 
		\bibliographystyle{amsplain}
		
	\end{footnotesize}
\end{document}